\documentclass[11pt]{article}
\pdfoutput=1
\usepackage{fullpage}
\usepackage{amsmath, amsthm}
\usepackage{bbm}
\usepackage{hyperref}
\usepackage{graphicx}
\usepackage{bm}
\usepackage{subcaption}
\usepackage{float}
\usepackage{listings}
\usepackage{geometry}
\usepackage[utf8x]{inputenc}
\usepackage{verbatim}
\usepackage{rotating}
\usepackage{lscape}
\usepackage{dsfont}
\usepackage{amsfonts}
\usepackage{color}
\usepackage{xr}
\usepackage{authblk}

\newtheorem{theorem}{Theorem}
\newtheorem{definition}{Definition}
\newtheorem{lemma}{Lemma}
\newtheorem{remark}{Remark}

\newtheorem{proposition}{Proposition}

\newcommand{\bas}[1]{\begin{align*}#1\end{align*}}
\newcommand{\ba}[1]{\begin{align}#1\end{align}}
\newcommand{\beq}[1]{\begin{equation}#1\end{equation}}
\newcommand{\bsplt}[1]{\begin{split}#1\end{split}}
\newtheorem{corollary}{Corollary}[theorem]

\newcommand{\cE}{\mathcal{E}}
\newcommand{\cF}{\mathcal{F}}
\newcommand{\cL}{\mathcal{L}}
\newcommand{\cN}{\mathcal{N}}

\newcommand{\cS}{\mathcal{S}}
\newcommand{\cX}{\mathcal{X}}
\newcommand{\cY}{\mathcal{Y}}

\newcommand{\bA}{\mathbb{A}}
\newcommand{\bB}{\mathbb{B}}
\newcommand{\bE}{\mathbb{E}}
\newcommand{\bR}{\mathbb{R}}

\newcommand{\Xt}{\tilde{X}}
\newcommand{\mut}{\tilde{\bmu}}
\newcommand{\bmu}{\bm{\mu}}
\newcommand{\ip}[1]{\langle #1 \rangle}
\newcommand{\diag}{\text{diag}}

\newcommand{\rmax}{R_{\max}}
\newcommand{\rmin}{\ensuremath R_{\min}}

\newcommand{\gamfun}[1]{\Gamma\left(#1\right)} 
\newcommand{\half}[1]{\frac{#1}{2}}
\newcommand{\rd}{\color{red}}

\newcommand{\bk}{\color{black}}

\newcommand{\norm}[1]{\left\lVert#1\right\rVert}
\newcommand{\bX}{\mathbf{X}}

\begin{document}

\title{
Convergence Analysis of Gradient EM \\
for Multi-component Gaussian Mixture
}
\author[1]{
Bowei Yan}
\author[1]{Mingzhang Yin}
\author[1]{Purnamrita Sarkar
}
\affil[1]{University of Texas at Austin}
\date{}
\maketitle

\begin{abstract}
In this paper, we study convergence properties of the gradient Expectation-Maximization algorithm~\cite{lange1995gradient} for Gaussian Mixture Models for general number of clusters and mixing coefficients. We derive the convergence rate depending on the mixing coefficients, minimum and maximum pairwise distances between the true centers and dimensionality and number of components; and obtain a near-optimal local contraction radius. While there have been some recent notable works that derive local convergence rates for EM in the two equal mixture symmetric GMM, in the more general case, the derivations need structurally different and non-trivial arguments. We use recent tools from learning theory and empirical processes to achieve our theoretical results.
\end{abstract}
\section{Introduction}

Proposed by \cite{dempster1977maximum} in 1977, the Expectation-Maximization (EM) algorithm is a powerful tool for statistical inference in latent variable models. 
A famous example is the parameter estimation problem under parametric mixture models. In such models, data  is generated from a mixture of a known family of parametric distributions. The mixture component from which a datapoint is generated from can be thought of as a latent variable.

Typically the marginal data log-likelihood (which integrates the latent variables out) is hard to optimize, and hence EM iteratively optimizes a lower bound of it and obtains a sequence of estimators. This consists of two steps. In the expectation step (E-step) one computes the expectation of the complete data likelihood with respect to the posterior distribution of the unobserved mixture memberships evaluated at the current parameter estimates. In the maximization step (M-step) one this expectation is maximized to obtain new estimators. 
EM 
always improves the objective function. While it is established in~\cite{conniffe1987expected} that the true parameter vector is the global maximizer of the log-likelihood function, there has been much effort to understand the behavior of the local optima obtained via EM.  

When the exact M-step is burdensome, a popular variant of EM, named Gradient EM is widely used. The idea here is to take a gradient step towards the maxima of the expectation computed in the E-step. \cite{lange1995gradient} introduces a gradient algorithm using one iteration of Newton's method and shows the local properties of the gradient EM are almost identical with those of the EM. 

Early literature \cite{wu1983convergence, xu1996convergence} mostly focuses on the convergence to the stationary points or local optima. In~\cite{wu1983convergence} it is proven that the sequence of estimators in EM converges to stationary point when the lower bound function from E-step is continuous. In addition, some conditions are derived under which EM converges to local maxima instead of saddle points; but these are typically hard to check. A link between EM and gradient methods is forged in~\cite{xu1996convergence} via a projection matrix and  the local convergence rate of EM is obtained. In particular, it is shown that for GMM with well-separated centers, the EM achieves faster convergence rates comparable to a quasi-Newton algorithm. While the convergence of EM deteriorates under worse separations, it is observed in \cite{redner1984mixture} that the mixture density determined by estimator sequence of EM reflects the sample data well.

In recent years, there has been a renewed wave of interest in studying the behavior of EM especially in GMMs. The global convergence of EM for a mixture of two equal-proportion Gaussian distributions is fully characterized in~\cite{xu2016global}. For more than two clusters, a negative result on EM and gradient EM being trapped in local minima arbitrarily far away from the global optimum is shown in~\cite{jin2016local}.

For high dimensional GMMs with $M$ components, the parameters are learned via reducing the dimensionality via a random projection in~\cite{dasgupta1999learning}. In~\cite{dasgupta2000two} the two-round method is proposed, where one first initializes with more than $M$ points, then prune to get one point in every cluster. It is pointed out in this paper that in high dimensional space, when the clusters are well separated, the mixing weight will go to either 0 or 1 after one single update. It is showed in \cite{yan2016robustness, mixon2016clustering} that one can cluster high dimensional sub-gaussian mixtures by semi-definite programming relaxations.

For the convergence rate of EM algorithm, it is observed in \cite{naim2012convergence} that a very small mixing proportion for one mixture component compared to others leads to slow convergence. \cite{balakrishnan2014statistical} gives non-asymptotic convergence guarantees in isotropic, balanced, two-component GMM; their result proves the linear convergence of EM if the center is initialized in a small neighborhood of the true parameters. The local convergence result in this paper has a sub-optimal contraction region. 

$K$-means clustering is another widely used clustering method. Lloyd's algorithm for $k$-means clustering has a similar flavor as EM. At each step, it recomputes the centroids of each cluster and updates the membership assignments alternatively. While EM does soft clustering at each step, Lloyd's algorithm obtains hard clustering. The clustering error of Lloyd's algorithm for arbitrary number of clusters is studied in~\cite{lu2016statistical}. The authors also show local convergence results where the contraction region is less restrictive than~\cite{balakrishnan2014statistical}. 

We would like to point out that there are many notable algorithms~\cite{kumar2010clustering,awasthi2012improved,vempala2004spectral} with provable guarantees for estimating mixture models. In~\cite{matouvsek2000approximate, friggstad2016local} the authors propose polynomial time algorithms which achieve epsilon approximation to the k-means loss.  A spectral algorithm for learning mixtures of gaussians is proposed in~\cite{vempala2004spectral}. We want to point out that our aim is not to come up with a new algorithm for mixture models, but to understand the interplay of model parameters in the convergence of gradient EM for a mixture of Gaussians with $M$ components. As we discuss later, our work also immediately leads to convergence guarantees of Stochastic Gradient EM. 
Another important difference is that the aim of these works is recovering the hidden mixture component memberships, whereas our goal is completely different: we are interested in understanding how well EM can estimate the mean parameters under a good initialization. 

In this paper, we study the convergence rate and local contraction radius of gradient EM under GMM with arbitrary number of clusters and mixing weights which are assumed to be known. For simplicity, we assume that the components share the same covariance matrix, which is known. Thus it suffices to carry out our analysis for isotropic GMMs with identity as the shared covariance matrix. We obtain a near-optimal condition on the contraction region in contrast to ~\cite{balakrishnan2014statistical}'s contraction radius for the mixture of two equal weight Gaussians.  We want to point out that, while the authors of~\cite{balakrishnan2014statistical} provide a general set of conditions to establish local convergence for a broad class of mixture models, the derivation of specific results and conditions on local convergence are tailored to the balance and symmetry of the model. 

We follow the same general route: first we obtain conditions for population gradient EM, where all sample averages are replaced by their expected counterpart. Then we translate the population version to the sample one. While the first part is conceptually similar, the general setting calls for more involved analysis. The second step typically makes use of concepts from empirical processes, by pairing up Ledoux-Talagrand contraction type arguments with well established symmetrization results. However, in our case, the function is not a contraction like in the symmetric two component case, since it involves the cluster estimates of all $M$ components.  Furthermore, the standard analysis of concentration inequalities by McDiarmid's inequality gets complicated because the bounded difference condition is not satisfied in our setting. We overcome these difficulties by taking advantage of recent tools in Rademacher averaging for vector valued function classes, and variants of McDiarmid type inequalities for functions which have bounded difference with high probability.

The rest of the paper is organized as follows. In Section \ref{sec:setup}, we state the problem and the notations. In Section 3, we provide the main results in local convergence rate and region for both population and sample-based gradient EM in GMMs. Section \ref{sec:local} and \ref{sec:sample_proof_sketch} provide the proof sketches of population and sample-based theoretical results, followed by the numerical result in Section \ref{sec:exp}. We conclude the paper with some discussions.

\section{Problem Setup and Notations}
\label{sec:setup}
Consider a GMM with $M$ clusters in $d$ dimensional space, with weights $\bm{\pi}=(\pi_1,\cdots,\pi_M)$.  Let $\bmu_i \in \bR^{d}$ be the mean of cluster $i$. Without loss of generality, we assume $\bE X=\sum_i \pi_i \bmu_i=0$ and the known covariance matrix for all components is $I_d$. Let $\bmu\in \bR^{Md}$ be the vector stacking the $\bmu_i$s vertically. We represent the mixture as $X\sim \text{GMM}(\pi,\bmu,I_d)$, which has the density function $p(x|\bmu)=\sum_{i=1}^M \pi_i\phi(x|\bmu_i, I_d)$.
where $\phi(x;\bmu,\Sigma)$ is the PDF of $N(\bmu,\Sigma)$. Then the population log-likelihood function as 
$
\cL(\bmu) = \bE_{X} \log \left( \sum_{i=1}^M \pi_i \phi(X|\bmu_i,I_d) \right)
$.
The Maximum Likelihood Estimator is then defined as $\hat{\bmu}_{\text{ML}} = \arg\max p(X|\bmu)$.
EM algorithm is based on using an auxiliary function to lower bound the log likelihood. Define
$
Q(\bmu|\bmu^t)=\bE_X \left[\sum_i p(Z=i|X;\bmu^t)\log \phi(X;\bmu_i,I_d)\right]
$, where $Z$ denote the unobserved component membership of data point $X$. 
The standard EM update is $\bmu^{t+1} = \arg\max_{\bmu} Q(\bmu|\bmu^t)$. Define 
\ba{
w_i(X;\bmu) = \frac{\pi_i\phi(X|\bmu_i,I_d) }{\sum_{j=1}^M \pi_j \phi(X|\bmu_j,I_d)}
\label{eq:def_w}
}
The update step for gradient EM, defined via the gradient operator $G(\bmu^t):\bR^{Md}\rightarrow \bR^{Md}$, is
\ba{
G(\bmu^t)^{(i)}:=\bmu_{i}^{t+1} = \bmu_i^t+s[\nabla Q(\bmu^t|\bmu^t)]_i = \bmu_i^t+s\bE_X \left[ \pi_i w_i(X;\bmu^t)(X-\bmu_i^t) \right].
\label{eq:grad_Q} }
where $s>0$ is the step size and $(.)^{(i)}$ denotes the part of the stacked vector corresponding to the $i^{th}$ mixture component. We will also use $G_n(\bmu)$ to denote the empirical counterpart of the population gradient operator $G(\mu)$ defined in Eq~\eqref{eq:grad_Q}. We assume we are given an initialization $\bmu^0_i$ and the true mixing weight $\pi_i$ for each component.
\subsection{Notations}
Define $\rmax$ and $\rmin$ as the largest and smallest distance between cluster centers i.e., $\rmax = \max_{i\ne j}\|\bmu_i^*-\bmu_j^*\|, \rmin = \min_{i\ne j}\|\bmu_i^*-\bmu_j^*\|$. Let $\pi_{\max}$ and $\pi_{\min}$ be the maximal and minimal cluster weights, and define $\kappa$ as $\kappa = \frac{\pi_{\max}}{\pi_{\min}}$.
Standard complexity analysis notation $o(\cdot), O(\cdot), \Theta(\cdot), \Omega(\cdot)$ will be used. $f(n)=\tilde{\Omega}(g(n))$ is short for $\Omega(g(n))$ ignoring logarithmic factors, equivalent to $f(n)\ge Cg(n)\log^k(g(n))$, similar for others. We use $\otimes$ to represent the kronecker product.

\section{Main Results}
\label{sec:mainres}
Despite being a non-convex problem, EM and gradient EM algorithms have been shown to exhibit good convergence behavior in practice, especially with good initializations. However, existing local convergence theory only applies for two-cluster equal-weight GMM. In this section, we present our main result in two parts. First we show the convergence rate and present a near-optimal radius for contraction region for population gradient EM. Then in the second part we connect the population version to finite sample results using concepts from empirical processes and learning theory.

\subsection{Local contraction for population gradient EM}
Intuitively, when $\bmu^t$ equals the ground truth $\bmu^*$, then the $Q(\bmu|\bmu^*)$ function will be well-behaved. This function is a key ingredient in \cite{balakrishnan2014statistical}, where the curvature of the $Q(\cdot|\bmu)$ function is shown to be close to the curvature of $Q(\cdot|\bmu^*)$ when the $\bmu$ is close to $\bmu^*$. This is a local property that only requires the gradient to be stable at one point.

\begin{definition}[Gradient Stability]
The Gradient Stability (GS) condition, denoted by $\text{GS}(\gamma, a)$, is satisfied if there exists $\gamma >0$, such that for $\bmu_i^t\in \bB(\bmu_i^*,a)$ with some $a>0$, for $\forall i\in[M]$.
\bas{
\|\nabla Q(\bmu^t|\bmu^*) - \nabla Q(\bmu^t|\bmu^t)\| \le \gamma \|\bmu^t-\bmu^*\|
}
\label{def:gs}
\end{definition}

The GS condition is used to prove contraction of the sequence of estimators produced by population gradient EM.  
However, for most latent variable models, it is typically challenging to verify the GS condition and obtain a tight bound on the parameter $\gamma$. 
We derive the GS condition under milder conditions (see Theorem~\ref{th:gmm_gs} in Section~\ref{sec:local}), which bounds the deviation of the partial gradient evaluated at $\bmu^t_i$ uniformly over all $i\in[M]$. This immediately implies the global GS condition defined in Definition \ref{def:gs}. Equipped with this result, we achieve a nearly optimal local convergence radius for general GMMs in Theorem \ref{th:population_contraction}. The proof of this theorem can be found in Appendix \ref{sec:proof_main_popu}.

\begin{theorem}[Convergence for Population gradient EM]
Let $d_0:=\min\{d,M\}$. If $\rmin=\tilde{\Omega}(\sqrt{d_0})$, with initialization $\bmu^0$ satisfying, $\|\bmu_i^0-\bmu_i^*\|\le a, \forall i\in [M]$, where
\bas{
	a \le& \half{\rmin}-\sqrt{d_0}O\left(\sqrt{\log\left(\max\left\{\frac{M^2\kappa}{\pi_{\min}},\rmax,d_0\right\}\right)}\right)
}

then the Population EM converges:
\bas{
\|\bmu^t-\bmu^*\| \le \zeta^t \|\bmu_0-\bmu^*\|,\qquad \zeta=\frac{\pi_{\max}-\pi_{\min}+2\gamma}{\pi_{\max}+\pi_{\min}}<1
} 
where $\gamma=M^{2}(2\kappa+4)\left( 2\rmax+ d_0 \right)^2\exp\left(-\left(\half{\rmin}-a\right)^2\sqrt{d_0}/8\right) < \pi_{\min}$.
\label{th:population_contraction}
\end{theorem}

\begin{remark}
The local contraction radius is largely improved compared to that in \cite{balakrishnan2014statistical}, which has $\rmin/8$ in the two equal sized symmetric GMM setting.
It can be seen that in Theorem \ref{th:population_contraction}, $a/\rmin$ goes to $\half{1}$ as the signal to noise ratio goes to infinity. 
We will show in simulations that when initialized from some point that lies $\rmin/2$ away from the true center, gradient EM only converges to a stationary point which is not a global optimum. More discussion can be found in Section~\ref{sec:exp}.
\end{remark}

\subsection{Finite sample bound for gradient EM} 
In the finite sample setting, as long as the deviation of the sample gradient from the population gradient is uniformly bounded, the convergence in the population setting implies the convergence in finite sample scenario. Thus the key ingredient in the proof is to get this uniform bound over all parameters in the contraction region $\bA$, i.e. bound $\sup_{\bmu\in \bA}\|G^{(i)}(\bmu)-G_n^{(i)}(\bmu)\|$, where $G$ and $G_n$ are defined  in Section~\ref{sec:setup}.

To prove the result, we expand the difference and define the following function for $i\in [M]$, where $u$ is a unit vector on a $d$ dimensional sphere $\cS^{d-1}$. This appears because we can write the Euclidean norm of any vector $B$,  as $\|B\|=\sup_{u\in \cS^{d-1}}\langle B, u\rangle$.
\ba{
	g_i^u(X) = \sup_{\bmu\in \bA} \frac{1}{n}\sum_{i=1}^n w_1(X_i;\bmu)\ip{X_i-\bmu_1,u}-\bE w_1(X;\bmu)\ip{X-\bmu_1,u}.
	\label{eq:def_g}
}

We will drop the super and subscript and prove results for $g_1^u$ without loss of generality.

The outline of the proof is to show that $g(X)$ is close to its expectation. This expectation can be further bounded via the Rademacher complexity of the corresponding function class (defined below in Eq~\eqref{def:function_class}) by the tools like the symmetrization lemma \cite{mohri2012foundations}. 

Consider the following class of functions indexed by $\bmu$ and some unit vector on $d$ dimensional sphere $u\in \cS^{d-1}$:
\ba{
	\cF_i^u = \{f^i: \cX\to \bR | f^i(X;\bmu,u) = w_i(X;\bmu)\ip{X-\bmu_i,u}\}
	\label{def:function_class}
}
We need to bound the $M$ functions classes separately for each mixture.
Given a finite $n$-sample $(X_1,\cdots, X_n)$, for each class, we define the Rademacher complexity as the expectation of empirical Rademacher complexity.
\bas{
	\hat{R}_n(\cF_i^u) = \bE_\epsilon \left[ \sup_{\bmu\in \bA}\frac{1}{n}\sum_{j=1}^n \epsilon_i f^i(X_j;\bmu,u) \right]; \qquad R_n(\cF_i^u)=\bE_{X} \hat{R}_n(\cF_i^u)
}
where $\epsilon_i$'s are the i.i.d. Rademacher random variables. 

For many function classes, the computation of the empirical Rademacher complexity can be hard. For complicated functions which are Lipschitz w.r.t functions from a simpler function class, one can use Ledoux-Talagrand type contraction results~\cite{ledoux2013probability}. In order to use the Ledoux-Talagrand contraction, one needs a 1-Lipschitz function, which we do not have, because our function involves $\bmu_i$, $i\in[M]$. Also, the weight functions $w_i$ are not separable in terms of the $\bmu_i$'s. 
Therefore the classical contraction lemma does not apply.
In our analysis, we need to introduce a vector-valued function, with each element involving only one $\bmu_i$, and apply a recent result of vector-versioned contraction lemma \cite{maurer2016vector}. With some careful analysis, we get the following. The details are deferred to Section~\ref{sec:sample_proof_sketch}.

\begin{proposition}
Let $\cF_i^u$ as defined in Eq.~\eqref{def:function_class} for $\forall i\in [M]$, then for some universal constant $c$,
\bas{
R_n(\cF_i^u) \le \frac{c M^{3/2}(1+\rmax)^3\sqrt{d}\max\{1,\log(\kappa)\}}{\sqrt{n}}
}
\label{prop:rademacher_bound}
\end{proposition}
After getting the Rademacher complexity, one needs to use concentration results like McDiarmid's inequality~\cite{mcdiarmid1989method} to achieve the finite-sample bound. 
Unfortunately for the functions defined in Eq.~\eqref{def:function_class}, the martingale difference sequence does not have bounded differences. Hence it is difficult to apply McDiarmid's inequality in its classical form. To resolve this, we instead use an extension of McDiarmid's inequality which can accommodate sequences which have bounded differences with high probability~\cite{combes2015extension}.

\begin{theorem}[Convergence for sample-based gradient EM]
Let $\zeta$
be the contraction parameter in Theorem~\ref{th:population_contraction}, and 
\ba{
\epsilon^{\text{unif}}(n)=  \tilde{O}(\max\{n^{-1/2}M^3(1+\rmax)^3\sqrt{d}\max\{1,\log(\kappa)\}, (1+\rmax)d/\sqrt{n}\}).
\label{eq:def_epsilon}
}
If $\epsilon^{\text{unif}}(n)\leq (1-\zeta)a$, then sample-based gradient EM satisfies 
\bas{
\norm{\hat{\bmu}_i^t-\bmu_i^*} \leq \zeta^t \norm{\bmu^0-\bmu^*}_2+\frac{1}{1-\zeta}\epsilon^{\text{unif}}(n); \quad \forall i\in [M]
}
with probability at least $1-n^{-cd}$, where $c$ is a positive constant.
\label{th:sample}
\end{theorem}

\begin{remark}
When data is observed in a streaming fashion, the gradient update can be modified into a stochastic gradient update, where the gradient is evaluated based on a single observation or a small batch.
By the GS condition proved in Theorem~\ref{th:population_contraction}, combined with Theorem 6 in \cite{balakrishnan2014statistical}, we immediately extend the guarantees of gradient EM into the guarantees for the stochastic gradient EM. 
\end{remark}

\subsection{Initialization}
Appropriate initialization for EM is the key to getting good estimation within fewer restarts in practice. There have been a number of interesting initialization algorithms for estimating mixture models.  
It is pointed out in \cite{jin2016local} that in practice, initializing the centers by uniformly drawing from the data is often more reasonable than drawing from a fixed distribution. Under this initialization strategy, we can bound the number of initializations required to find a ``good'' initialization that falls in the contraction region in Theorem~\ref{th:population_contraction}. 
 The exact theorem statement and a discussion of random initialization can be found in Appendix~\ref{app:initialization}. More sophisticated strategy includes, an approximate solution to $k$-means on a projected low-dimensional space used in~\cite{awasthi2012improved} and~\cite{kumar2010clustering}.  While it would be interesting to study different initialization schemes, that is part of future work.

\section{Local Convergence of Population Gradient EM}
\label{sec:local}
In this section we present the proof sketch for Theorem \ref{th:population_contraction}. The complete proofs in this section are deferred to Appendix \ref{app:proof_popu}. To start with, we calculate the closed-form characterization of the gradient of $q(\bmu)$ as stated in the following lemma. 

\begin{lemma}
Define $q(\bmu) = Q(\bmu|\bmu^*)$. The gradient of $q(\bmu)$ is
$\nabla q(\bmu) = (\diag(\pi)\otimes I_d) \left( \bmu^*-\bmu \right)$.
\label{lem:grad_q}
\end{lemma}

If we know the parameter $\gamma$ in the gradient stability condition, then the convergence rate depends only on the condition number of the Hessian of $q(\cdot)$ and $\gamma$.
\begin{theorem}[Convergence rate for population gradient EM]
	If $Q$ satisfies the GS condition with parameter $0<\gamma<\pi_{\min}$, denote $d_t:=\norm{\bmu_t-\bmu^*}$, then with step size $s=\frac{2}{\pi_{\min}+\pi_{\max}}$, we have:
	\bas{
		d_{t+1} \le \left(\frac{\pi_{\max}-\pi_{\min}+2\gamma}{\pi_{\max}+\pi_{\min}}\right)^t d_{0}
	}   
    \label{th:contraction_with_gamma}
\end{theorem}
The proof uses an approximation on gradient and standard techniques in analysis of gradient descent. 
\begin{remark}
It can be verified that the convergence rate is equivalent to that shown in \cite{balakrishnan2014statistical} when applied to GMMs. 
The convergence slows down as the proportion imbalance $\kappa = \pi_{\max}/\pi_{\min}$ increases, which matches the observation in \cite{naim2012convergence}. 
\end{remark}

Now to verify the GS condition, we have the following theorem.

\begin{theorem}[GS condition for general GMM]
If $\rmin=\tilde{\Omega}(\sqrt{\min\{d,M\}})$, and $\bmu_i \in \bB(\bmu_i^*,a), \forall i\in [M]$ where 
$a \leq \half{\rmin} - \sqrt{\min\{d,M\}}\max(4\sqrt{2[\log (\rmin/4)]_+},8\sqrt{3})$,
then $\|\nabla_{\bmu_i} Q(\bmu|\bmu^t) - \nabla_{\bmu_i} q(\bmu)\| \le \frac{\gamma}{M} \sum_{i=1}^M\|\bmu_i^t-\bmu_i^*\| \le \frac{\gamma}{\sqrt{M}}\|\bmu^t-\bmu^*\|$,\\
where $\gamma = M^{2}(2\kappa+4)\left( 2\rmax+ \min\{d,M\} \right)^2\exp\left(-\left(\half{\rmin}-a\right)^2\sqrt{\min\{d,M\}}/8\right)$. \\
Furthermore,
$
\|\nabla Q(\bmu|\bmu^t) - \nabla q(\bmu)\| \le \gamma \|\bmu^t-\bmu^*\|
$.
\label{th:gmm_gs}
\end{theorem}

\begin{proof}[Proof sketch of Theorem \ref{th:gmm_gs}]
W.l.o.g. we show the proof with the first cluster, consider the difference of the gradient corresponding to $\bmu_1$.
\beq{
\bsplt{
\nabla_{\bmu_1} Q(\bmu^t|\bmu^t) - \nabla_{\bmu_1} q(\bmu^t)  = &  \bE (w_1(X;\bmu^t)- w_1(X;\bmu^*))(X-\bmu_1^t)
}
\label{eq:gs_lhs_step1}
}
For any given $X$, consider the function $\bmu \to w_1(X;\bmu)$, we have
\beq{
\bsplt{
\nabla_{\bmu} w_1(X;\bmu) =&
\begin{pmatrix} w_1(X;\bmu) (1-w_1(X;\bmu) )(X-\bmu_1)^T \\
- w_1(X;\bmu) w_2(X;\bmu) (X-\bmu_2)^T\\
\vdots\\
- w_1(X;\bmu) w_M(X;\bmu) (X-\bmu_M)^T
\end{pmatrix}
}
\label{eq:diff_w1}
}
Let $\bmu^u = \bmu^*+u(\bmu^t-\bmu^*), \forall u\in [0,1]$, obviously $\bmu^u \in \Pi_{i=1}^M \bB(\bmu_i^*, \|\bmu_i^t-\bmu_i^*\|) \subset \Pi_{i=1}^M \bB(\bmu_i^*, a)$.
By Taylor's theorem,
\beq{
\bsplt{
&\| \bE (w_1(X;\bmu_1^t) -w_1(X;\bmu_1^*))(X-\bmu_1^t)\| = \left\Vert\bE \left[\int_{u=0}^1 \nabla_{u} w_1(X;\bmu^u) du (X-\bmu_1^t)  \right]\right\Vert\\
 \le &  U_1 \|\bmu_1^t-\bmu_1^*\|_2 + \sum_{i\ne 1}U_i\|\bmu_i^t-\bmu_i^*\|_2 \le \max_{i\in [M]}\{U_i\} \sum_i \|\bmu_i^t-\bmu_i^*\|_2
}
\label{eq:u1_plus_ui_main}
}
where 
\bas{
U_1 =& \sup_{u\in [0,1]} \| \bE w_1(X;\bmu^u) (1-w_1(X;\bmu^u) ) (X-\bmu_1^t)(X-\bmu_{1}^u)^T \|_{op}\\
U_i =& \sup_{u\in [0,1]} \| \bE w_1(X;\bmu^u) w_i(X;\bmu^u) (X-\bmu_1^t)(X-\bmu_{2}^u)^T \|_{op}
}
Bounding them with careful analysis on Gaussian distribution 
yields the result. The technical details are deferred to Appendix \ref{app:proof_popu}.
\end{proof}

\section{Sample-based Convergence}
\label{sec:sample_proof_sketch}

In this section we present the proof sketch for sample-based convergence of gradient EM. 
The main ingredient in proof of Theorem~\ref{th:sample} is the result of the following theorem, which develops an uniform upper bound for the differences between sample-based gradient and population gradient on each cluster center. 
\begin{theorem}[Uniform bound for sample-based gradient EM]
Denote $\mathbb{A}$ as the contraction region $\Pi_{i=1}^M\bB(\bmu_i^*,a)$.  Under the condition of Theorem \ref{th:population_contraction}, with probability at least $1-\exp\left( -cd \log n\right)$, 
\bas{
\sup_{\bmu \in \mathbb{A}}\norm{G^{(i)}(\bmu)-G_n^{(i)}(\bmu)}<\epsilon^{\text{unif}}(n); \qquad \forall i\in [M]
}
where $\epsilon^{\text{unif}}(n)$ is as defined in Eq.~\eqref{eq:def_epsilon}.
\label{th:sample_eps}
\end{theorem}
\begin{remark}
	It is worth pointing out that, the first part of the bound on $\epsilon^{\text{unif}(n)}$ in Eq.~\eqref{eq:def_epsilon} comes from the Rademacher complexity, which is optimal in terms of the order of $n$ and $d$. However the extra factor of $\sqrt{d}$ and $\log(n)$ comes from the altered McDiarmid's inequality, tightening which will be left for future work.
\end{remark}

\begin{proof}[Proof sketch of Theorem~\ref{th:sample_eps}]
	
	Denote $Z_i=\sup_{\bmu \in \mathbb{A}}\norm{G^{(i)}(\bmu)-G^{(i)}_n(\bmu)}$. Recall $g_i^u(X)$ defined in Eq.~\eqref{eq:def_g}, then it is not hard to see that $Z_i=\sup_{u\in \cS^{d-1}} g_i^u(X)$. Consider a $\frac{1}{2}$-covering $\{u^{(1)},\cdots ,u^{(K)}\}$ of the unit sphere $\cS^{d-1}$, where $K$ is the covering number of an unit sphere in $d$ dimensions. We can show that $Z_i \le 2 \max_{j=1,\cdots, K} g_i^{u^{(j)}}(X)$. 

	As we state below in Lemma~\ref{lem:zuj_minus_ezuj}, we have for each $u$, with probability at least $1-\exp\left(-cd\log n\right)$, 
	$ g_i^{u}(X) = \tilde{O}(\max\{ R_n(\cF_i^u), (1+\rmax)d/\sqrt{n} \})$. Plugging in the Rademacher complexity from Proposition \ref{prop:rademacher_bound}, standard bounds on $K$, and applying union bound,
	we have
	\bas{
		Z_i\le \tilde{O}(\max\{ n^{-1/2}M^3(1+\rmax)^3\sqrt{d}\max\{1,\log(\kappa)\}, (1+\rmax)d/\sqrt{n}\})
	}
	with probability at least $1-\exp\left(2d-cd\log n\right)=1-\exp\left(-c'd\log n\right)$. 
\end{proof}

Iteratively applying Theorem~\ref{th:sample_eps}, we can bound the error in $\bmu$ for the sample updates. The full proof is deferred to Appendix~\ref{app:proof_sample}. The key step is the following lemma, which bounds the $g_i^u(X)$ for any given $u\in \cS^{d-1}$. Without loss of generality we prove the result for $i=1$.

\begin{lemma}
	Let $u$ be a unit vector. $X_i, i=1,\cdots,n$ are i.i.d. samples from $\text{GMM}(\pi,\bmu^*,I_d)$. $g_1^u(X)$ as defined in Eq.~\eqref{eq:def_g}. Then with probability $1-\exp(-cd\log n)$ for some constant $c>0$, 
	$ g_1^u(X)=\tilde{O}(\max\{ R_n(\cF_1^u), (1+\rmax)d/\sqrt{n} \})$.
	\label{lem:zuj_minus_ezuj}
\end{lemma}

The quantity $g_1^u(X)$ depends on the sample, the idea for proving Lemma~\ref{lem:zuj_minus_ezuj} is to show it concentrates around its expectation when sample size is large. 
Note that when the function class has bounded differences (changing one data point changes the function by a bounded amount almost surely), as in the case in many risk minimization problems in supervised learning, the McDiarmid's inequality can be used to achieve concentration. 
However the function class we define in Eq.~\eqref{def:function_class} is not bounded almost everywhere, but with high probability, hence the classical result does not apply. 
Conditioning on the event where the difference is bounded, we use an extension of McDiarmid's inequality by \cite{combes2015extension}. For convenience, we restate a weaker version of the theorem using our notation below.

\begin{theorem}[\cite{combes2015extension}]
Consider independent random variable $X=(X_1,\cdots, X_n)$ in the product probability space $\cX=\prod_{i=1}^n \cX_i$, where $\cX_i$ is the probability space for $X_i$. Also consider a function $g:\cX\to \bR$. If there exists a subset $\cY\subset \cX$, and a scalar $c>0$, such that
\bas{
|g(x)-g(y)|\le L, \forall x,y\in \cY, x_j=y_j,\forall j\ne i.
}
Denote $p=1-P(X\in \cY)$, then 
$P(g(X) - \bE[g(X)|X\in \cY]\ge \epsilon) \le p+\exp\left( -\frac{2(\epsilon-npL)_+^2}{nL^2} \right).$
\label{th:extension_mcdiarmid}
\end{theorem}

It is worth pointing out that in Theorem~\ref{th:extension_mcdiarmid}, the concentration is shown in reference to the conditional expectation $\bE[g(X)|X\in \cY]$ when the data points are in the bounded difference set. So to fully achieve the type of bound given by McDiarmid's inequality, we need to further bound the difference of the conditional expectation and the full expectation. Combining the two parts we will be able to show that, the empirical difference is upper bounded using the Rademacher complexity.

Note that with a different martingale concentration result, we can achieve a tighter upper bound for the empirical difference with a cleaner proof. We defer the analysis in Appendix~\ref{app:tighter_sample_bound}.

Now it remains to derive the Rademacher complexity under the given function class. Note that when the function class is a contraction, or Lipschitz with respect to another function (usually of a simpler form), one can use the Ledoux-Talagrand contraction lemma \cite{ledoux2013probability} to reduce the Rademacher complexity of the original function class to the Rademacher complexity of the simpler function class. This is essential in getting the Rademacher complexities for complicated function classes. As we mention in Section \ref{sec:mainres}, our function class in Eq.~\eqref{def:function_class} is unfortunately not Lipschitz due to the fact that it involves all cluster centers even for the gradient on one cluster. We get around this problem by introducing a vector valued function, and show that the functions in Eq.~\eqref{def:function_class} are Lipschitz in terms of the vector-valued function. In other words, the absolute difference in the function when the parameter changes is upper bounded by the norm of the vector difference of the vector-valued function. Then we build upon the recent vector-contraction result from \cite{maurer2016vector}, and prove the following lemma under our setting. 

\begin{lemma}
Let $X$ be nontrivial, symmetric and sub-gaussian. Then there exists a constant $C<\infty$, depending only on the distribution of $X$, such that for any subset $\cS$ of a separable Banach space and function $h_i:\cS \to \bR$, $f_i: \cS\to \bR^k$, $i\in [n]$ satisfying
$ \forall s,s'\in \cS, |h_i(s)-h_i(s')|\le L \|f(s)-f(s')\|$. If $\epsilon_{ik}$ is an independent doubly indexed Rademacher sequence, 
we have,
\bas{
\bE \sup_{s\in \cS}\sum_i \epsilon_i h_i(s)\le \bE \sqrt{2}L \sup_{s\in \cS}\sum_{i,k}\epsilon_{ik}f_i(s)_k,
}
where $f_i(s)_k$ is the $k$-th component of $f_i(s)$.
\label{lem:vector_contraction_separable}
\end{lemma}

\begin{remark}
In contrast to the original form in~\cite{maurer2016vector}, we have a $\cS$ as a subset of a separable Banach Space. The proof uses 
standard tools from measure theory, and is to be found in Appendix~\ref{app:proof_sample}. 
\end{remark}

This equips us to prove Proposition~\ref{prop:rademacher_bound}. 
\begin{proof}[Proof sketch of Proposition~\ref{prop:rademacher_bound}]
For any unit vector $u$, the Rademacher complexity of $\cF_1^u$ is
\beq{
\bsplt{
R_n(\cF_1^u) =& \bE_X \bE_\epsilon \sup_{\bmu\in \bA} \frac{1}{n}\sum_{i=1}^n \epsilon_i w_1(X_i;\bmu)\ip{X_i-\bmu_1,u}\\
\le & \underbrace{\bE_X \bE_\epsilon \sup_{\bmu\in \bA}\frac{1}{n}\sum_{i=1}^n \epsilon_i w_1(X_i;\bmu)\ip{X_i,u}}_{(D)} +  \underbrace{\bE_X \bE_\epsilon \sup_{\bmu\in \bA} \frac{1}{n}\sum_{i=1}^n \epsilon_iw_1(X_i;\bmu)\ip{\bmu_1,u}}_{(E)}
}
\label{eq:mgf_cauchy_schwarz}
}
We bound the two terms separately. 
Define $\eta_j(\bmu): \mathbb{R}^{Md}\to \mathbb{R}^{M}$ to be a vector valued function with the $k$-th coordinate  
\bas{
[\eta_j(\bmu)]_k=\frac{\norm{\bmu_1}^2}{2}-\frac{\norm{\bmu_k}^2}{2}+\langle X_j,\bmu_k-\bmu_1\rangle+\log\left(\frac{\pi_k}{\pi_1}\right)
}

\begin{flalign}
&\text{ It can be shown that } \qquad |w_1(X_j;\bmu)-w_1(X_j;\bmu')|\leq \frac{\sqrt{M}}{4}\norm{\eta_j(\bmu)-\eta_j(\bmu')}&
\label{eq:lipschitz}
\end{flalign}

Now let $\psi_1(X_j;\bmu)=w_1(X_j;\bmu)\langle X_j,u \rangle$. 
With Lipschitz property \eqref{eq:lipschitz} and Lemma~\ref{lem:vector_contraction}, 
we have
\bas{\bE \left[{\sup_{\bmu \in \mathbb{A}}\frac{1}{n}\sum_{j=1}^n\epsilon_j w_i(X_j;\bmu)\langle X_j,u \rangle}\right] \le 
\bE \left[ \frac{\sqrt{2}\sqrt{M}}{4n}\sup_{\bmu \in \mathbb{A}}\sum_{j=1}^n\sum_{k=1}^M \epsilon_{jk}[\eta_j(\bmu)]_k  \right]
}
The right hand side can be bounded with tools regarding independent sum of sub-gaussian random variables. Similar techniques apply to the $(E)$ term. Adding things up we get the final bound. 
\end{proof}

\section{Experiments}
\label{sec:exp}
In this section we collect some numerical results.
In all experiments we set the covariance matrix for each mixture component as identity matrix $I_d$ and define signal-to-noise ratio (SNR) as $R_{\min}$.

{\bf Convergence Rate}
We first evaluate the convergence rate and compare with those given in Theorem \ref{th:contraction_with_gamma} and Theorem \ref{th:gmm_gs}. 
For this set of experiments, we use a mixture of 3 Gaussians in 2 dimensions. In both experiments $\rmax/\rmin=1.5$.
In different settings of $\bm{\pi}$, we apply gradient EM with varying SNR from 1 to 5. For each choice of SNR, we perform 10 independent trials with $N=12,000$ data points. The average of $\log\norm{\bmu^t-\hat{\bmu}}$ and the standard deviation are plotted versus iterations. In Figure~\ref{fig:snr} (a) and (b) we plot balanced $\bm{\pi}$ ($\kappa=1$) and unbalanced $\bm{\pi}$ ($\kappa>1$) respectively.

All settings indicate the linear convergence rate as shown in Theorem \ref{th:contraction_with_gamma}. As SNR grows, the parameter $\gamma$ in GS condition decreases and thus yields faster convergence rate. Comparing left two panels in Figure \ref{fig:snr}, increasing imbalance of cluster weights $\kappa$ slows down the local convergence rate as shown in Theorem \ref{th:contraction_with_gamma}.

\begin{figure}[H]
\centering
\begin{tabular}{cccc}
\hspace{-1em}
\includegraphics[width=0.22\textwidth, height=2.9 cm]{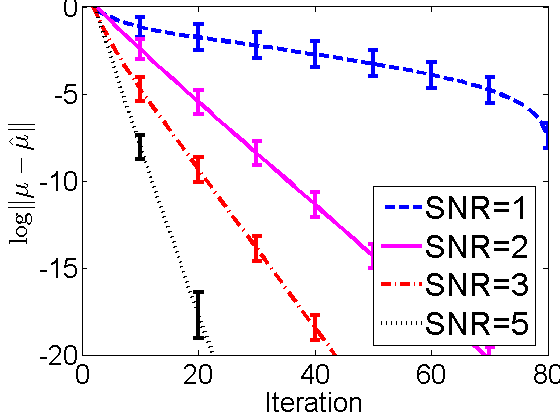}\hspace{-1em}&
\includegraphics[width=0.22\textwidth, height= 2.9 cm]{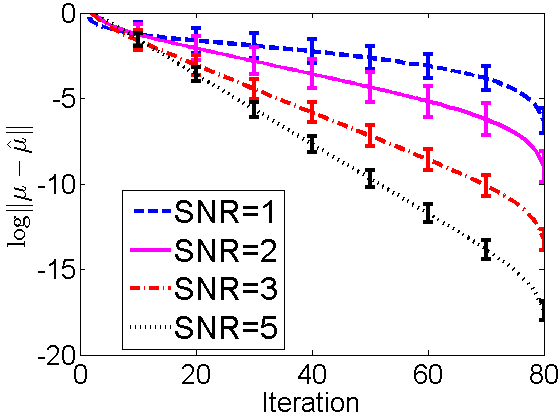}\hspace{-1em}&
\hspace{-1em}\includegraphics[width=0.27\textwidth]{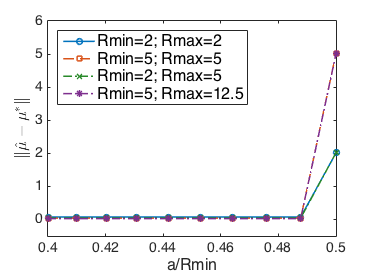}&
\hspace{-1em}\includegraphics[width=0.26\textwidth, height = 3.1cm]{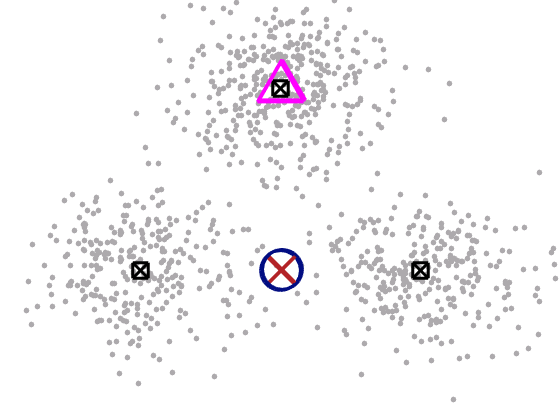}
\hspace{-2em}\\
(a) & (b) & (c)& (d)
\end{tabular}
\caption{(a, b): The influence of SNR on optimization error in different settings. The figures represent the influence of SNR when the GMMs have different cluster centers and weights: (a) $\bm{\pi}=(1/3,1/3,1/3)$. (b) $\bm{\pi}=(0.6,0.3,0.1)$.  (c) plots statistical error with different initializations arbitrarily close to the boundary of the contraction region. (d) shows the suboptimal stationary point when two centers are initialized from the midpoint of the respective true cluster centers.
}
\label{fig:snr}
\end{figure}

{\bf Contraction Region}
To show the tightness of the contraction region, we generate a mixture with $M=3, d=2$, and initialize the clusters as follows. We use $\bmu^0_2=\half{\bmu^*_2+\bmu^*_3}-\epsilon$, $\bmu^0_3=\half{\bmu^*_2+\bmu^*_3}+\epsilon$, for shrinking $\epsilon$, i.e. increasing $a/\rmin$ and plot the error on the Y axis. Figure \ref{fig:snr}-(c) shows that gradient EM converges when initialized arbitrarily close to the boundary, thus confirming our near optimal contraction region.  Figure \ref{fig:snr}-(d) shows that when $\epsilon=0$, i.e. $a=\half{\rmin}$, gradient EM can be trapped at a sub-optimal stationary point. 

\section{Concluding Remarks}
In this paper, we obtain local convergence rates and a near optimal contraction radius for population and sample-based gradient EM for GMM with general cluster number and weights. For our sample based analysis, we face new challenges because of structural differences from the two component equal weight setting, which are alleviated via the use of non-traditional tools like a vector valued contraction argument and martingale concentrations bounds where bounded differences hold only with high probability. 

\appendix
\section{Accompanying Lemmas}
In this subsection, we collect some lemmas on Gaussian distribution and basic properties of Gaussian mixture model. Most of them can be derived with fundamental analysis techniques.
The following lemma from \cite{vershynin2010introduction} bounds the covering number of a unit sphere.
\begin{lemma}[Lemma 5.2 \cite{vershynin2010introduction}]
Let $\cS^{n-1}$ be the unit Euclidean sphere equipped with Euclidean metric. Denote $\cN(\cS^{n-1},\epsilon)$ as the covering number with $\epsilon$-net, then
\bas{
\cN(\cS^{n-1},\epsilon) \le \left(1+\frac{2}{\epsilon}\right)^n
}
Specifically, when $\epsilon=1/2$, we have
\bas{
\cN(\cS^{n-1},\half{1})\le \exp(2n)
}
\label{lem:covering_number}
\end{lemma}

The following lemma is useful while carrying out spherical coordinate transformation.
\begin{lemma}
\begin{itemize}
\item[(1)] The volume for a $d$-dimensional $r$-ball is $\frac{\pi^{\half{d}}}{\gamfun{\half{d}+1}}r^d$;
\item[(2)]
$\int_0^\pi \sin^k(x)dx = \frac{\sqrt{\pi}\Gamma(\frac{k+1}{2})}{\Gamma\left( \frac{k}{2}+1\right)}$, and
\bas{
 \int_{\theta_{d-1}=0}^{2\pi}\int_{\theta_{d-2}=0}^{\pi}\cdots \int_{\theta_1=0}^{\pi} \sin^{d-2}(\theta_1)\cdots \sin(\theta_{d-2}) d\theta_1\cdots d\theta_{d-1} =\frac{2\pi^{\frac{d}{2}}}{\Gamma\left(\frac{d}{2}\right)}
}
\item[(3)]
If $X\sim \cN(\bmu,\sigma^2I_d)$, then 
\bas{
\bE_X \|X-\bmu\|^p = 2^{\half{p}}\frac{\gamfun{\half{p+d}}}{\gamfun{\half{d}}}\sigma^p
}
\end{itemize}
\label{lem:normal_norm}
\end{lemma}

\begin{proof}
(1, 2) can be proven by elementary integration. Now we prove (3).
By spherical coordinate transformation,
\bas{
\bE_X \|X-\bmu\|^p =& (2\pi\sigma^2)^{-\frac{d}{2}}\int_{u=0}^\infty u^{p+d-1}e^{-\frac{u^2}{2\sigma^2}}du \frac{2\pi^{\frac{d}{2}}}{\Gamma\left(\frac{d}{2}\right)}
= 2^{\half{p}}\frac{\gamfun{\half{p+d}}}{\gamfun{\half{d}}}\sigma^p
}
\end{proof}

\begin{lemma}[Gamma tail bound \cite{boucheron2013concentration}]
If $X\sim \text{Gamma}(v,c)$, then $P(X>\sqrt{2vt}+ct)\le e^{-t}$.
Or equivalently,
\bas{
P(X>t)\le \exp\left( -\frac{v}{c^2}\left( 1+\frac{ct}{v}-\sqrt{1+\frac{2ct}{v}} \right) \right)
}
In particular, if $\frac{ct}{v}\ge 4$,
\bas{
P(X>t)\le \exp\left( -\frac{v}{c^2}\sqrt{\frac{ct}{v}} \right) = \exp\left( -\sqrt{\frac{vt}{c^3}} \right)
}
\label{lem:gamma_tail}
\end{lemma}

\begin{lemma}
For $\forall d>0$, if $r \ge 2\sqrt{d+1}$, then
\bas{
\int_r^\infty u^d e^{-\frac{u^2}{2}}du \le  2^{\frac{d-1}{2}}\gamfun{\frac{d+1}{2}} \exp\left( -\frac{r}{2}\sqrt{d+1} \right)
}
For $p\in \{0,1,2\}$, when $r\ge 2\sqrt{d+p}$,
\bas{
\int_r^\infty (u+x)^{p}u^{d-1} e^{-\frac{u^2}{2}}du \le  2^{\frac{d}{2}-1}\gamfun{\frac{d}{2}}(x+d)^p \exp\left( -\frac{r}{2}\sqrt{d} \right)
}
\label{lem:pol_exp_tail_bound}
\end{lemma}

\begin{proof}
By changing of variables $v=\half{u^2}$ and integration by parts, we have 
\bas{
\int_r^\infty u^d e^{-\frac{u^2}{2}}du =& 2^{\frac{d-1}{2}}\int_{\half{r^2}}^\infty v^{\frac{d-1}{2}} e^{-v}dv\\
=& 2^{\frac{d-1}{2}}\gamfun{\frac{d+1}{2}} P(V>\half{r^2})
}
where $V\sim \text{Gamma}(\frac{d+1}{2},1)$. By Lemma \ref{lem:gamma_tail}, if $r^2\ge 4(1+d)$, 
\bas{
P\left(V> \half{r^2}\right) \le \exp\left( -\frac{r}{2}\sqrt{d+1} \right)
}
Hence we have the first inequality. For the second, when $p=0$, it follows directly from first part. When $p=1$,
\bas{
\int_r^\infty (u+x)^{p}u^{d-1} e^{-\frac{u^2}{2}}du =& 
\int_r^\infty u^d e^{-\frac{u^2}{2}}du +x\int_r^\infty u^{d-1} e^{-\frac{u^2}{2}}du \\
\le & 2^{\half{d-1}}\gamfun{\half{d+1}}\exp\left(-\half{r}\sqrt{d+1}\right)+x2^{\half{d}-1}\gamfun{\half{d}}\exp\left(-\half{r}\sqrt{d}\right)\\
\le & 2^{\frac{d}{2}-1}\gamfun{\frac{d}{2}}(x+d) \exp\left( -\frac{r}{2}\sqrt{d} \right)
}
where we use $\gamfun{\half{d+1}}<\gamfun{\half{d}+1}=\half{d}\gamfun{\half{d}}$, and $\exp\left(-\half{r}\sqrt{d+1}\right)<\exp\left(-\half{r}\sqrt{d}\right)$ in the last step.

When $p=2$,
\bas{
&\int_r^\infty (u+x)^2 u^{d-1} e^{-\frac{u^2}{2}}du =
\int_r^\infty u^{d+1} e^{-\frac{u^2}{2}}du+2x\int_r^\infty u^{d} e^{-\frac{u^2}{2}}du \\
&+x^2\int_r^\infty u^{d-1} e^{-\frac{u^2}{2}}du \\
\le & 2^{\half{d}}\gamfun{\half{d}+1}\exp\left(-\half{r}\sqrt{d+2}\right) + 2x\cdot 2^{\half{d-1}}\gamfun{\half{d+1}}\exp\left(-\half{r}\sqrt{d+1}\right)+x^2 2^{\half{d}-1}\gamfun{\half{d}}\exp\left(-\half{r}\sqrt{d}\right)\\
\le & (d+\sqrt{2}dx+x^2) 2^{\frac{d}{2}-1}\gamfun{\frac{d}{2}} \exp\left( -\frac{r}{2}\sqrt{d} \right)\\
\le & (x+d)^2 2^{\frac{d}{2}-1}\gamfun{\frac{d}{2}} \exp\left( -\frac{r}{2}\sqrt{d} \right)
}
\end{proof}

Using Lemma~\ref{lem:pol_exp_tail_bound}, we can get an easy to use tail bound for Euclidean norm of a Gaussian vector.
\begin{lemma}
If $X\sim \cN(0,I_d)$, for $r\ge 2\sqrt{d}$, we have 
\bas{
P(\|X\| \ge r) \le \exp(-\frac{r\sqrt{d}}{2})
}
\label{lem:norm_of_gaussian}
\end{lemma}
\begin{proof}
By spherical coordinate transformation,
\bas{
P(\|X\|\ge r) =& \int (2\pi)^{-d/2} \exp(-\|x\|^2/2) dx\\
=& (2\pi)^{-d/2} \frac{2\pi^{d/2}}{\gamfun{\half{d}}} \int_r^\infty r^{d-1}e^{-r^2/2}dr\\
\le & \exp\left(-\half{r}\sqrt{d}\right)
}
\end{proof}

\begin{lemma}
If $X\sim \text{GMM}(\pi,\bmu^*,\sigma^2 I_d)$, then $X$ is a sub-gaussian random vector with sub-gaussian norm $\sigma+\sum_{i=1}^M \pi_i \|\bmu_i^*\|$.
\label{lem:mixture_sub_gaussian}
\end{lemma}
\begin{proof}

For any unit vector $u$, consider the random variable $X_u = \ip{X,u}$. By the definition in \cite{vershynin2010introduction}, it suffices to show that $X_u$ has a sub-gaussian norm upper bounded by $\sigma+\sum_{i=1}^M \pi_i \|\bmu_i^*\|$.
\bas{
\|X_u\|_{\phi_2} = \sup_{p\ge 1} (\bE |X_u|^p)^{1/p}
}
For any $p\ge 1$, let $Z$ be the latent variable in the mixture model, we have 
\bas{
p^{-1/2}\left( \bE |X_u|^p\right)^{1/p} =& p^{-1/2}\left(\sum_{i=1}^M \bE [|X_u|^p|Z=i]\cdot P(Z=i)\right)^{1/p}\\
\le & p^{-1/2}\sum_{i=1}^M \pi_i \left(\bE [|X_u|^p|Z=i]\right)^{1/p}\\
\stackrel{(i)}{\le} & p^{-1/2} \sum_{i=1}^M \pi_i \left(\bE [ |X_u-\bmu_i^*|^p|Z=i]^{1/p}+\|\bmu_i^*\|\right) \\
\le & p^{-1/2}\left( \sum_{i=1}^M \pi_i p^{1/2}\sigma + \|\bmu_i^*\|\right) \le \sigma  +\sum_{i=1}^M \pi_i \|\bmu_i^*\|
}
where $(i)$ follows from Minkovski's inequality.
\end{proof}

The following lemma characterize the relation between $\|\bmu^*_{\max}\|$ and $\rmax$.
\begin{lemma}
If $X\sim \text{GMM}(\pi,\bmu^*,\sigma^2 I_d)$ with $\bE X=0$, let $\|\bmu^*_{\max}\| = \max_i \|\bmu_i^*\|$, then 
\bas{
\|\bmu^*_{\max}\|\le \rmax \le 2\|\bmu^*_{\max}\|
}
\label{lem:rmax_mumax}
\end{lemma}
\begin{proof}
We first prove $\|\bmu^*_{\max}\|\le \rmax $ by contradiction. Assume $\|\bmu^*_{\max}\| > \rmax$, by definition of $\rmax$, all the cluster centers lies in the ball $\bB(\|\bmu^*_{\max}\|,\rmax)$, but the origin is outside of the ball, which contradicts the fact that $\bE X=\sum_i \pi_i \bmu_i^*=0$.

The second inequality follows from triangle inequality, assume $\rmax$ is achieved at $R_{ij}$, then
\bas{
\rmax \le \|\bmu_i^*\|+\|\bmu_j^*\| \le 2\|\bmu^*_{\max}\|.
}
\end{proof}
\begin{lemma}
A function $f:\mathbb{R}^n \to \mathbb{R}$ is $\sqrt{n}L$ Lipschitz if there exists a constant $L$ such that the restriction of $f$ on a certain coordinate is $L$-Lipschitz.
\label{lem:vec_Lipschitz}
\end{lemma}
\begin{proof}
We first relax the norm of difference via a chain of triangle inequalities where each pair of terms only vary on one dimension.
\bas{
& |f(x_1,x_2,\cdots, x_n) - f(y_1,y_2, \cdots, x_n)| \\
 \leq &  \sum_{i=1}^n |f(y_1,y_2,\cdots, y_{i-1},x_i,x_{i+1}, \cdots, x_n) - f(y_1,y_2,\cdots, y_{i-1},y_i, x_{i+1}, \cdots, x_n)| \\
\leq &  \sum_{i=1}^n L |x_i - y_i | \leq \sqrt{n}L\norm{x-y}
}

\end{proof}

\section{Proofs in Section~\ref{sec:local}}
\label{app:proof_popu}

\begin{proof}[Proof of Lemma~\ref{lem:grad_q}]
By \eqref{eq:grad_Q}, $\nabla_{\bmu_i} q(\bmu) = \bE_X w_i(X;\bmu^*)(X-\bmu_i)$. Without loss of generality, we only show the claim for $i=1$. That is equivalent of saying, if $X\sim \text{GMM}(\pi,\bmu^*)$, we have 
$ \bE [w_1(X;\bmu^*)(X-\bmu_1^*)] = 0$.
Denote $\cN(\bmu_i^*, \Sigma)$ as $\cN_i$ and its distribution as $\phi_i(X)$.
Decompose the left hand side with respect to the mixture components, we have
\bas{
\bE [w_1(X)X] =& \sum_i \pi_i \bE_{X\sim \cN_i} [w_1(X)X]\\
=& \sum_i \pi_i \int \phi_i(X)\frac{\pi_1\phi_1(X)}{\sum_k \pi_k{\phi_k(X)}}Xdx\\
=& \pi_1\bE_{X\sim \cN_1}X = \pi_1\bmu_1^*
}

Similarly $\bE [w_1(X)]=\pi_1$. Hence $\nabla_{\bmu_1} q(\bmu) = \bE_X w_1(X;\bmu^*)(X-\bmu_1)= \pi_1(\bmu_1^*-\bmu_1)$.\\
This completes the proof.
\end{proof}

\begin{proof}[Proof of Theorem~\ref{th:contraction_with_gamma}]
	Define 
	By Lemma \ref{lem:grad_q}, the GS condition is equivalent to 
	\bas{
		\norm{\nabla Q(\bmu|\bmu^t)-\nabla q(\bmu)} \le \gamma \|\bmu^t-\bmu^*\|
	}
By triangle inequality,  
    \bas{
    \norm{\bmu_1^{t+1}-\bmu_1^*} =&
    \norm{\bmu_1^{t}-\bmu_1^*+s\nabla Q(\bmu|\bmu^t)}\\
    \le & \norm{\bmu_1^{t}-\bmu_1^*+s\nabla q(\bmu)}+s\norm{\nabla Q(\bmu|\bmu^t)-\nabla q(\bmu)}\\
    \le & \frac{\pi_{\max}-\pi_{\min}}{\pi_{\max}+\pi_{\min}}\norm{\bmu_1^{t}-\bmu_1^*}+\frac{2}{\pi_{\max}+\pi_{\min}}\gamma \norm{\bmu_1^{t}-\bmu_1^*}\\
    \le & \frac{\pi_{\max}-\pi_{\min}+2\gamma}{\pi_{\max}+\pi_{\min}}\norm{\bmu_1^{t}-\bmu_1^*}
    }
    To see why the last inequality hold, notice that $q(\bmu)$ has largest eigenvalue $-\pi_{\min}$ and smallest eigenvalue $-\pi_{\max}$. Apply the classical result for gradient descent, with step size $s=\frac{2}{\pi_{\max}+\pi_{\min}}$ guarantees 
    \bas{
    \norm{\bmu_1^{t}-\bmu_1^*+s\nabla q(\bmu)}\le \frac{\pi_{\max}-\pi_{\min}}{\pi_{\max}+\pi_{\min}}\norm{\bmu_1^{t}-\bmu_1^*}
    }
\end{proof}

\subsection{Proofs of Theorem \ref{th:gmm_gs}}

We start with two lemmas.
\begin{lemma}
For $X\sim \text{GMM}(\pi,\bmu^*,I_d)$, if $\rmin=\tilde{\Omega}(\sqrt{d})$, and $\bmu_i \in \bB(\bmu_i^*,a), \forall i\in [M]$ where 
\bas{
a \leq \half{\rmin} - \sqrt{d}\max(4\sqrt{2[\log (\rmin/4)]_+},8\sqrt{3}).
}
Then for $p=0,1,2$ and $\forall i\in[M]$, we have
\bas{
\bE_X w_i(X;\bmu)(1-w_i(X;\bmu))\|X-\bmu_i\|^p \le 2M\left(\half{3}\rmax +d\right)^{p}\exp \left(-\left(\half{\rmin}-a\right)^2\sqrt{d}/8 \right)
}

\label{lem:poly_expectation}
\end{lemma}

Using the same techniques, for the cross terms, we have the following lemma.
\begin{lemma}
Assume $X\sim \text{GMM}(\pi,\bmu^*,I_d)$, and $\bmu_i \in \bB(\bmu_i^*,a), \forall i\in [M]$. Under the same conditions as in Lemma~\ref{lem:poly_expectation}, we have for $\forall i\ne j \in [M]$, 
\bas{
\bE_{X} [w_i(X;\bmu)w_j(X;\bmu)\|X-\bmu_i\|\cdot \|X-\bmu_j\|] 
\le &  (1+2\kappa)\left(\frac{3}{2}\rmax+d\right)^2\exp\left(-\left(\half{\rmin}-a\right)^2\sqrt{d}/8\right)  
}
\label{lem:wiwj_bound}
\end{lemma}
\begin{proof}[Proof of Lemma \ref{lem:poly_expectation}]
Without loss of generality, we prove the claim for $i=1$. Recall the definition of $w_i(X;\bmu)$ from Equation~\ref{eq:def_w}. For $p\in \{0, 1,2\}$,
\beq{
\bsplt{
& \bE_{X} w_1(X;\bmu)(1-w_1(X;\bmu))\|X-\bmu_1\|^p \\
=& \sum_{i\in [M]} \pi_i \bE_{X\sim \cN(\bmu_i^*)} w_1(X;\bmu)(1-w_1(X;\bmu)) \|X-\bmu_1\|^p\\
 \leq &\pi_1 \bE_{X\sim \cN(\bmu_1^*)}w_1(X;\bmu)(1-w_1(X;\bmu))\|X-\bmu_1\|^p + \sum_{i\ne 1} \pi_i \bE_{X\sim \cN(\bmu_i^*)} w_1(X;\bmu) \|X-\bmu_1\|^p
}
\label{eq:e_mixture_wx}
}
First let us look at the first term. 
Define event $\cE_r^{(1)}=\{X: X\sim \cN(\bmu_1^*); \|X-\bmu_1^*\|\le r\}$ for some $r>0$. We will see later that we need $r<\half{\rmin}-a$. Then for $X\in \cE_r^{(1)}$ using triangle inequality, we have 
\ba{
\|X-\bmu_i\| \begin{cases}
\leq \|X-\bmu_i^*\|+\|\bmu_i^*-\bmu_i\|\leq r+a & \mbox{$i=1$}\\
\geq \|\bmu_i-\bmu_1^*\|-\|X-\bmu_1^*\|\geq \|\bmu_i^*-\bmu_1^*\|-\|\bmu_i^*-\bmu_i\|-r\geq \rmin-r-a& \mbox{$i\neq 1$}
\end{cases}
}
\bas{
& \bE_{X\sim \cN(\bmu_1^*)}w_1(X;\bmu)(1-w_1(X;\bmu))\|X-\bmu_1\|^p\\
=& \bE[w_1(X;\bmu)(1-w_1(X;\bmu))\|X-\bmu_1\|^p | \cE_r^{(1)}]P(\cE_r^{(1)})\\
& \qquad \qquad +\bE[w_1(X;\bmu)(1-w_1(X;\bmu))\|X-\bmu_1\|^p | \cE_r^{(1)c}]P(\cE_r^{(1)c})
}
In view of the fact that $w_1(X;\bmu)$ is monotonically decreasing w.r.t. $\|X-\bmu_i\|$ and increasing w.r.t. $\|X-\bmu_1\|$, we have
\bas{
1-w_1(X;\bmu)\le & \frac{(1-\pi_1)\exp\left( -\half{(\rmin-r-a)^2} \right)}{\pi_1 \exp\left(-\half{(r+a)^2}\right)+(1-\pi_1)\exp\left( -\half{(\rmin-r-a)^2} \right)}\\
\le & \frac{1-\pi_1}{\pi_1} \exp\left(-\half{1}\rmin(\rmin-2r-2a)\right)
}
Also notice that $w_1(X;\bmu)\le 1$, we have
\bas{
& \bE[w_1(X;\bmu)(1-w_1(X;\bmu))\|X-\bmu_1\|^p | \cE_r^{(1)}]P(\cE_r^{(1)})\\
\le & \frac{1-\pi_1}{\pi_1} \exp\left(-\half{1}\rmin(\rmin-2r-2a)\right) (r+a)^p
}

For $\cE_r^{(1)c}$, note $w_1(X;\bmu)(1-w_1(X;\bmu))\le \frac{1}{4}$, we have for $p=1$,
\bas{
&\bE[w_1(X;\bmu)(1-w_1(X;\bmu))\|X-\bmu_1\| | \cE_r^{(1)c}]P(\cE_r^{(1)c})\\
\le & \frac{1}{4}\int_{u=r}^\infty (u+a) (2\pi)^{-\half{d}}\exp\left(-\half{u^2}\right)\cdot  \frac{2\pi^{\half{d}}}{\gamfun{\half{d}}}u^{d-1} du\\
\le & \frac{1}{4}(2\pi)^{-\half{d}}\frac{2\pi^{\half{d}}}{\gamfun{\half{d}}}\int_{u=r}^\infty (u+a) \exp\left(-\half{u^2}\right)u^{d-1} du\\
\stackrel{(i)}{\le} & \frac{a+d}{4}\exp\left(-\half{r}\sqrt{d}\right)
}
The inequality (i) follows from Lemma \ref{lem:pol_exp_tail_bound} when $r>2\sqrt{d+1}$.
Similarly, for $p=2$,
\bas{
&\bE[w_1(X;\bmu)(1-w_1(X;\bmu))\|X-\bmu_1\|^2 | \cE_r^{(1)c}]P(\cE_r^{(1)c})\\
\le & \frac{2^{-\half{d}-1}}{\gamfun{\half{d}}} \int_r^\infty (u+a)^2u^{d-1}e^{-\half{u^2}}du \stackrel{(ii)}{\le} \frac{(a+d)^2}{4}\exp\left(-\half{r}\sqrt{d}\right)
}
The inequality (ii) follows from Lemma \ref{lem:pol_exp_tail_bound} when $r>2\sqrt{d+1}$ and $p=2$.
Therefore for the first mixture we have,
\ba{
& \pi_1 \bE_{X\sim \cN(\bmu_1^*)}w_1(X;\bmu)(1-w_1(X;\bmu))\|X-\bmu_1\|^p \nonumber\\
\le & (1-\pi_1) (r+a)^p\exp\left( -\half{1}\rmin(\rmin-2r-2a) \right) +\pi_1 \frac{(a+d)^p}{4}\exp\left(-\half{r}\sqrt{d}\right)
\label{eq:w1mw_firstmix}
}
Next we bound $\bE_{X\sim \cN(\bmu_i^*)} w_1(X;\bmu) \|X-\bmu_1\|^p$ for $i\ne 1$. For some $0<r<\half{R}-a$, we have
\beq{
\bsplt{
&\pi_i \bE_{X\sim \cN(\bmu_i^*)} w_1(X;\bmu)\|X-\bmu_1\|^p \\
= & \int_X \frac{\pi_1\phi(X;\bmu_1)\cdot \pi_i\phi(X;\bmu_i^*)}{\sum_j \pi_j \phi(X;\bmu_j)}\|X-\bmu_1\|^p dX\\
=& \underbrace{\int_{X\in \bB(\bmu_i^*,r)} \frac{\pi_1\phi(X;\bmu_1)\cdot \pi_i\phi(X;\bmu_i^*)}{\sum_j \pi_j\phi(X;\bmu_j)}\|X-\bmu_1\|^p dX}_{I_1^{(p)}}+\underbrace{\int_{X\not\in \bB(\bmu_i^*,r)} \frac{\pi_1\phi(X;\bmu_1) \cdot \pi_i\phi(X;\bmu_i^*)}{\sum_j \pi_j\phi(X;\bmu_j)}\|X-\bmu_1\|^p dX}_{I_2^{(p)}}
}
\label{eq:piiEwnorm}
}
 When $\|X-\bmu_i^*\|\le r$, since by assumption $\|\bmu_i-\bmu_i^*\|\le a$, 
\beq{
\bsplt{
\frac{\phi(X;\bmu_i^*)}{\phi(X;\bmu_i)} = & \exp\left(\frac{\|X-\bmu_i\|^2}{2}-\frac{\|X-\bmu_i^*\|^2}{2}\right)\\
=&  \exp\left( \left(X-\frac{\bmu_i+\bmu_i^*}{2}\right)^T(\bmu_i-\bmu_i^*)\right) 
}
}
Since by Cauchy-Schwarz we have $|(X-\frac{\bmu_i+\bmu_i^*}{2})^T(\bmu_i-\bmu_i^*)|=|(X-\bmu_i^*+\frac{\bmu_i^*-\bmu_i}{2})^T(\bmu_i-\bmu_i^*)|\leq (r+a/2)a $, we have:
\ba{
\exp\left(-(r+\half{a})a\right)\leq \frac{\phi(X;\bmu_i^*)}{\phi(X;\bmu_i)}\leq \exp\left((r+\half{a})a\right)
\label{eq:ratio_of_phi_ra}
}
For such $X$, $\phi(X;\bmu_1) \le (2\pi)^{-\half{d}}\exp\left(-\half{(\rmin-r-a)^2}\right)$, and we have 
\bas{
I_1^{(p)} =& \int_{X\in \bB(\bmu_i^*,r)} \frac{\pi_1\phi(X;\bmu_1)\pi_i\phi(X;\bmu_i^*)}{\sum_j \pi_j\phi(X;\bmu_j)}\|X-\bmu_1\|^p dX\\
\le & \int_{X\in \bB(\bmu_i^*,r)} \frac{\pi_1\phi(X;\bmu_1)\pi_i\phi(X;\bmu_i)\exp\left((r+\half{a})a\right)}{\sum_j\pi_j \phi(X;\bmu_j)}\|X-\bmu_1\|^p dX\\
\le & \pi_1\exp\left((r+\half{a})a\right)\int_{X\in \bB(\bmu_i^*,r)} \phi(X;\bmu_1)\|X-\bmu_1\|^p dX\\
\le & \pi_1(2\pi)^{-d/2}\exp\left((r+\half{a})a\right)(\rmax+a+r)^p\exp\left(-\half{(\rmin-r-a)^2}\right) \frac{\pi^{d/2}}{\Gamma(\half{d}+1)}r^d\\
\le & \frac{\pi_1 2^{-d/2}}{\Gamma(\half{d}+1)}\exp\left((r+\half{a})a-\frac{(\rmin-r-a)^2}{2}\right)(\rmax+a+r)^{p}r^d\\
\le & \pi_1 2^{1-d}\exp\left(\rmin\left(a- \half{\rmin}(1-r/\rmin)^2\right)\right)(\rmax+a+r)^{p}r^d\\
}
The last inequality follows from the fact that $\gamfun{\half{d}+1}\ge ([\half{d}])!\ge 2^{\half{d}-1}$.
On the other hand, for $I_2$, since $w_1(X;\bmu)\le 1$, taking spherical coordinate transformation we have,
\bas{
I_2^{(p)} \le &  \int_{\|X-\bmu_i^*\|\ge r} \pi_i \phi(X;\bmu_i^*)\|X-\bmu_1\|^p dX\\
\le & \pi_i\int_{\|X-\bmu_i^*\|\ge r} (2\pi)^{-d/2}\exp(-\half{\|X-\bmu_i^*\|^2}) \|X-\bmu_1\|^p dX\\
\le &  \frac{\pi_i 2^{1-d/2}}{\Gamma(\frac{d}{2})} \int_{u=r}^\infty u^{d-1}\exp \left(-\half{u^2}\right) (u+\rmax+a)^p du 
}
Apply Lemma \ref{lem:pol_exp_tail_bound}, when $r\ge 2\sqrt{d+2}$, for $p\in\{0,1,2\}$
\ba{
I_2^{(p)} \le &  \pi_i \left(\rmax+a+d\right)^p\exp\left(-\half{r}\sqrt{d}\right)
\label{eq:outside_ball_quadratic_bound}
}
Summing up $I_1$ and $I_2$, for any $0<r<\rmin/2$, from \eqref{eq:piiEwnorm} we get:
\ba{
& \pi_i \bE_{X\sim \cN(\bmu_i^*)} w_1(X;\bmu)\|X-\bmu_1\|^p \nonumber \\
\le & \pi_1 2^{1-d}\exp\left(\rmin\left(a- \half{\rmin}(1-r/\rmin)^2\right)\right)(\rmax+a+r)^{p}r^d + \pi_i \left(\rmax+a+d\right)^p \exp\left(-\half{r}\sqrt{d}\right) \label{eq:piiw1norm}
}
Now plugging Eq.~\eqref{eq:w1mw_firstmix} and Eq.~\eqref{eq:piiw1norm} into Eq.~\eqref{eq:e_mixture_wx} gives, 
\bas{
&\bE_X w_1(X;\bmu)(1-w_1(X;\bmu))\|X-\bmu_1\|^p \\
\le &  (1-\pi_1) (r+a)^p\exp\left( -\half{1}\rmin(\rmin-2r-2a) \right) +\pi_1 \frac{(a+d)^p}{4}\exp\left(-\half{r}\sqrt{d}\right) \\
&+ \pi_1 (M-1)2^{1-d}\exp\left(\rmin\left(a- \half{\rmin}(1-r/\rmin)^2\right)\right)(\rmax+a+r)^{p}r^d\\
&+(1-\pi_1)\left(\rmax+a+d\right)^p \exp\left(-\half{r}\sqrt{d}\right)\\
\stackrel{}{\le} & \underbrace{(1-\pi_1) (r+a)^p\exp\left( -\half{1}\rmin(\rmin-2r-2a) \right)}_{(A)}
+\underbrace{\left(\rmax+a+d\right)^p \exp\left(-\half{r}\sqrt{d}\right)}_{(B)}\\
&+\underbrace{ 2\pi_1 (M-1) \exp\left(\rmin\left(a- \half{\rmin}(1-r/\rmin)^2\right)+d\log (r/2)\right)(\rmax+a+r)^{p}}_{(C)}
}

Note that in order to have a negative term inside exponential of (A),
we require $r+a< \half{\rmin}$. In order to ensure the same for (C), we need:
\ba{
\label{eq:a-ub}
a< \half{\rmin}\left(1-\frac{r}{\rmin}\right)^2
}

If $r^2\geq 2d\log (r/2)$, then we have:
\bas{
&\exp\left(\rmin\left(a- \half{\rmin}(1-r/\rmin)^2\right)+d\log (r/2)\right)\leq \exp\left(\rmin\left(a- \half{\rmin}(1-r/\rmin)^2\right)+r^2/2\right)\\
&\leq \exp\left(\rmin a- \left(\half{\rmin^2}-r\rmin+\half{r^2}\right)+\half{r^2}\right)\\
&=\exp\left( -\half{1}\rmin(\rmin-2r-2a) \right)
}
Therefore, $(A)+(C)\leq (1-\pi_1+2\pi_1(M-1))(R_{\max}+a+r)^{p}\exp\left( -\half{1}\rmin(\rmin-2r-2a) \right)$

Finally, if $r\leq \rmin\frac{\rmin/2-a}{\rmin+\sqrt{d}/2}$,
we have:
\bas{
\exp\left( -\half{1}\rmin(\rmin-2r-2a) \right)\leq \exp(-\half{r}\sqrt{d})
}
Hence, 
\bas{
(A)+(B)+(C)\leq & (2-\pi_1+2\pi_1(M-1))\left(\half{3}\rmax+d\right)^{p}\exp \left(-\half{r}\sqrt{d}\right)\\
\le& 2M\left(\half{3}\rmax+d\right)^{p}\exp \left(-\half{r}\sqrt{d}\right)
}

Set 
\ba{
r= \frac{\rmin/2-a}{4}, \quad a\le \frac{\rmin}{2}\label{eq:constraint13}
}
then Eq~\eqref{eq:a-ub} and $a+r\le \half{\rmin}$ are automatically satisfied. When $\rmin \ge \frac{\sqrt{d}}{6}$, we have $r\le \rmin\frac{\rmin/2-a}{\rmin+\sqrt{d}/2}$. Finally in order to meet the constraints
\ba{
r&\geq 2\sqrt{d+2} \Leftarrow r\geq 3\sqrt{d}\label{eq:constraint1}\\
r^2&\geq 2d\log r/2 \label{eq:constraint2}
}
we need
\bas{
\frac{\rmin/2-a}{4}&\geq \max(\sqrt{2d[\log (\rmin/4)]_+},2\sqrt{3}\sqrt{d})\\
a & \leq \half{\rmin}- \sqrt{d}\max(4\sqrt{2[\log( \rmin/4)]_+},8\sqrt{3})
}
The right hand side of last inequality is non-negative when $\rmin=\tilde{\Omega}(\sqrt{d})$.
Under these conditions, with Eq.~\eqref{eq:constraint13} plugged in, we have
\bas{
\bE_X w_1(X;\bmu)(1-w_1(X;\bmu))\|X-\bmu_1\|^p 
\le  2M\left(\half{3}\rmax+d\right)^{p}\exp \left(-\left(\half{\rmin}-a\right)^2\sqrt{d}/8 \right)
}
\end{proof}

\begin{proof}[Proof of Lemma \ref{lem:wiwj_bound}]
For any $r\le \half{\rmin}-a$, define $\cE_0=\{X: \exists i, \text{ such that } Z_X = i, \|X-\bmu_i^*\|> r\}$ and $\cE_k=\{X: Z_X=k, \|X-\bmu_k^*\|\le r\}$.
\bas{
& \bE_{X}\left[ w_i(X;\bmu)w_j(X;\bmu)\|X-\bmu_i\|\cdot \|X-\bmu_j\|\right] \\
\le & \underbrace{\bE_{X}\left[ w_i(X;\bmu)w_j(X;\bmu)\|X-\bmu_i\|\|X-\bmu_j\||\cE_0\right] P(\cE_0)}_{I_0} \\
&+ \sum_{k\in[M]} \underbrace{\pi_k \bE_{X\sim \cN(\bmu_k^*)} \left[w_i(X;\bmu)w_j(X;\bmu)\|X-\bmu_i\|\|X-\bmu_j\|| \|X-\bmu_k\|\le r \right]}_{I_k}
}
First we look at $I_0$, this again can be decomposed as the sum over mixtures. Similarly as in Eq.~\eqref{eq:outside_ball_quadratic_bound}, we have
\bas{
I_0 \le & \left(\rmax+a+d\right)^2\exp\left(-\half{r}\sqrt{d}\right)
}
For $I_k$, by Eq.~\eqref{eq:ratio_of_phi_ra},
\beq{
\bsplt{
I_k = & \int_X \frac{\pi_i\phi(X;\bmu_i)\pi_j\phi(X;\bmu_j)\pi_k\phi(X;\bmu_k^*)}{(\sum_\ell \pi_\ell \phi(X;\bmu_t))^2} \|X-\bmu_i\|\cdot \|X-\bmu_j\|dX\\
\le & \int_X \frac{\pi_i\phi(X;\bmu_i)\pi_j\phi(X;\bmu_j)\pi_k\phi(X;\bmu_k)\exp((r+a/2)a)}{(\sum_\ell \pi_\ell \phi(X;\bmu_\ell))^2} \|X-\bmu_i\|\cdot \|X-\bmu_j\|dX\\
\le & \kappa \pi_k 2\pi^{-\frac{d}{2}}\exp(-\frac{R_{(\min}-r-a)^2}{2}) \exp((r+a/2)a) (\rmax+r+a)^2 \frac{\pi^{d/2}}{\gamfun{\half{d}+1}}r^d \\
\le & \pi_k \kappa 2^{-d/2} \frac{1}{\gamfun{\half{d}+1}}r^d \exp\left((r+a/2)a-\half{(\rmin-r-a)^2}\right) (\rmax+r+a)^2 \\
\le & 2\pi_k \kappa  \exp\left(\rmin\left( a-\half{\rmin}\left( 1-\frac{r}{\rmin} \right)^2 \right) + d\log (r/2)\right) (\rmax+r+a)^2 
}
}
Adding up $I_k$'s and $I_0$, we have
\bas{
& \bE_{X} \left[w_i(X;\bmu)w_j(X;\bmu)\|X-\bmu_i\|\|X-\bmu_j\|\right] \\
\le & \left(\rmax+a+d\right)^2\exp\left(-\half{r}\sqrt{d}\right) \\
&\qquad \qquad+ 2\kappa \exp\left(\rmin\left( a-\half{\rmin}\left( 1-\frac{r}{\rmin} \right)^2 \right) + d\log (r/2)\right) (\rmax+r+a)^2 
}
Take $r=\frac{1}{4}\left(\half{\rmin}-a\right)$,
we have $\rmin\left( a-\half{\rmin}\left( 1-\frac{r}{\rmin} \right)^2 \right) + d\log (r/2) \le -\frac{r}{2}\sqrt{d}$. Therefore,
\bas{
& \bE_{X} [w_i(X;\bmu)w_j(X;\bmu)\|X-\bmu_i\|\cdot \|X-\bmu_j\|] \\
\le & (1+2\kappa)\left(\frac{3}{2}\rmax+d\right)^2\exp\left(-\left(\half{\rmin}-a\right)^2\sqrt{d}/8\right)  
}
\end{proof}

\begin{proof}[Proof of Theorem \ref{th:gmm_gs}]
Consider the difference of the gradient corresponding to $\bmu_i$, without loss of generality, assume $i=1$.
\beq{
\bsplt{
\nabla_{\bmu_1} Q(\bmu^t|\bmu^t) - \nabla q(\bmu^t)  = &  \bE (w_1(X;\bmu^t)- w_1(X;\bmu^*))(X-\bmu_1^t)
}
\label{eq:gs_lhs_step1}
}
For any given $X$, consider the function $\bmu \to w_1(X;\bmu)$, we have
\beq{
\bsplt{
\nabla_{\bmu} w_1(X;\bmu) =&
\begin{pmatrix} w_1(X;\bmu) (1-w_1(X;\bmu) )(X-\bmu_1)^T \\
- w_1(X;\bmu) w_2(X;\bmu) (X-\bmu_2)^T\\
\vdots\\
- w_1(X;\bmu) w_M(X;\bmu) (X-\bmu_M)^T
\end{pmatrix}
}
\label{eq:diff_w1}
}
Let $\bmu^u = \bmu^*+u(\bmu^t-\bmu^*), \forall u\in [0,1]$, obviously $\bmu^u \in \otimes_{i=1}^M \bB(\bmu_i^*, \|\bmu_i^t-\bmu_i^*\|) \subset \otimes_{i=1}^M \bB(\bmu_i^*, a)$.
 By Taylor's theorem,
\beq{
\bsplt{
&\| \bE (w_1(X;\bmu_1^t) -w_1(X;\bmu_1^*))(X-\bmu_1^t)\| = \left\Vert\bE \left[\int_{u=0}^1 \nabla_{u} w_1(X;\bmu^u) du (X-\bmu_1^t)  \right]\right\Vert\\
 =& \left\Vert \int_{u=0}^1 \bE w_1(X;\bmu^u) (1-w_1(X;\bmu^u) )(X-\bmu_{1}^u)^T(\bmu_1^t-\bmu_1^*) (X-\bmu_1^t)du \right.\\
 & \left. \qquad - \sum_{i\ne 1} \int_{u=0}^1 \bE w_1(X;\bmu^u) w_i(X;\bmu^u) )(X-\bmu_{2}^u)^T(\bmu_2^t-\bmu_2^*)(X-\bmu_1^t)du \right\Vert \\
 \le &  U_1 \|\bmu_1^t-\bmu_1^*\|_2 + \sum_{i\ne 1}U_i\|\bmu_i^t-\bmu_i^*\|_2
}
\label{eq:u1_plus_ui}
}
where 
\bas{
U_1 =& \sup_{u\in [0,1]} \| \bE w_1(X;\bmu^u) (1-w_1(X;\bmu^u) ) (X-\bmu_1^t)(X-\bmu_{1}^u)^T \|_{op}\\
U_i =& \sup_{u\in [0,1]} \| \bE w_1(X;\bmu^u) w_i(X;\bmu^u) (X-\bmu_1^t)(X-\bmu_{2}^u)^T \|_{op}
}
For $U_1$ by triangle inequality we have,
\ba{
U_1\le & \sup_{u\in [0,1]} \| \bE w_1(X;\bmu^u) (1-w_1(X;\bmu^u) ) (X-\bmu_{1}^u)(X-\bmu_{1}^u)^T \|_{op} \nonumber\\
&+\sup_{u\in [0,1]} \| \bE w_1(X;\bmu^u) (1-w_1(X;\bmu^u) ) (\bmu^u_{1}-\bmu_1^t)(X-\bmu^u_{1})^T \|_{op}\nonumber\\
\le & \sup_{u\in [0,1]} \| \bE w_1(X;\bmu^u) (1-w_1(X;\bmu^u) ) (X-\bmu_{1}^u)(X-\bmu_{1}^u)^T \|_{op} \nonumber\\
&+ a \sup_{u\in [0,1]} \| \bE w_1(X;\bmu^u) (1-w_1(X;\bmu^u) ) (X-\bmu^u_{1}) \|
\label{eq:U1_sup_opnorm}
}
We now develop an uniform bound for the operator norm. 
For any $u\in [0,1]$, there exists a rotation matrix $O$, such that all $R\bmu^u_i, i\in [M]$ have non-zero entries in the leading $\min\{d,M\}$ coordinates, and zeros for the remaining $[d-M]_+$ coordinates. 
Denote $\tilde{X}:=OX$, then $\tilde{X}|Z=i\sim \cN(O\bmu_i^*,I_d)$. Let 
\bas{
O\bmu_i^u=[\tilde{\bmu}_i^u,0_{[d-M]_+}]\text{ and }O \bmu_i^*=[v_i^{\min\{d,M\}},v_i^{[d-M]_+}],\quad \tilde{\bmu}_i^u\in \bR^{\min\{d,M\}}
}
For ease of notation, we assume $d\ge M$ for now, the other case can be derived without much modification. We can rewrite
\bas{
(X-\bmu_{1}^u)(X-\bmu_{1}^u)^T = O^T \begin{bmatrix}
(\tilde{X}^M-\tilde{\bmu_1^u})(\tilde{X}^M-\tilde{\bmu_1^u})^T & (\tilde{X}^M-\tilde{\bmu_1^u})(\tilde{X}^{d-M})^T\\
(\tilde{X}^{d-M})(\tilde{X}^M-\tilde{\bmu_1^u})^T & (\tilde{X}^{d-M})(\tilde{X}^{d-M})^T
\end{bmatrix} O
}

Note by the rotation, $w_i(X;\bmu)$ only depend on the first $M$ coordinates. And by isotropicity, $\Xt^M$ and $\Xt^{d-M}$ are independent. By $\bE \Xt^{d-M}=0$ (since we assume that the centroid of the means is at zero, and a rotation does not change that) and $\bE \Xt^{d-M}(\Xt^{d-M})^T=I_{d-M}+\sum_i \pi_i (v_i^{d-M})(v_i^{d-M})^T$, we have,

\bas{
& \| \bE w_1(X;\bmu^u) (1-w_1(X;\bmu^u) ) (X-\bmu_{1}^u)(X-\bmu_{1}^u)^T \|_{op}= \left\Vert \begin{bmatrix} D_1 & 0\\
0 &  D_2
\end{bmatrix}\right\Vert_{op}\\
\le & \max\{ \|D_1\|_{op}, \|D_2 \|_{op} \}
}
$D_1$ and $D_2$ are defined below.
Applying Lemma \ref{lem:poly_expectation} with dimension $\min\{d,M\}$, when $\rmin = \Omega(\sqrt{\min\{d,M\}})$,
\bas{
\|D_1\|_{op}&=\|\bE w_1(\Xt;\mut^u) (1-w_1(\Xt;\mut^u) ) (\tilde{X}^{\min\{d,M\}}-\tilde{\bmu_1^u})(\tilde{X}^{\min\{d,M\}}-\tilde{\bmu_1^u})^T\|_{op} \\
\le & 2M\left(\frac{3}{2}\rmax + \min\{d,M\}\right)^2\exp\left(-\left(\half{\rmin}-a\right)^2\sqrt{\min\{d,M\}}/8\right)
}

For $D_2$, by independence and Lemma \ref{lem:poly_expectation}, when $\rmin = \Omega(\sqrt{\min\{d,M\}})$,
\bas{
&\|D_2\|_{op} = \left\Vert \bE w_1(\Xt;\mut^u) (1-w_1(\Xt;\mut^u) )\left(I_{[d-M]_+}+\sum_i \pi_i (v_i^{[d-M]_+})(v_i^{[d-M]_+})^T\right)\right\Vert_{op}\\
=&  \left\Vert \left(\bE_{\Xt_{\min\{d,M\}}} w_1(\Xt_{\min\{d,M\}};\mut^u) (1-w_1(\Xt_{\min\{d,M\}};\mut^u) )\right) \right. \\
& \left. \qquad \qquad \cdot \bE_{X_{[d-M]_+}}\left(I_{[d-M]_+}+\sum_i \pi_i (v_i^{[d-M]_+})(v_i^{[d-M]_+})^T\right)\right\Vert_{op}\\
\le& (R_{\max}^2+1)2M\exp\left(-\left(\half{\rmin}-a\right)^2\sqrt{\min\{d,M\}}/8\right)
}
Combining the two and plugging in Eq.~\eqref{eq:U1_sup_opnorm},
\bas{
U_1 \le & 2M\exp\left(-\left(\half{\rmin}-a\right)^2\sqrt{\min\{d,M\}}/8\right)\cdot\\
&\qquad \left( \max\left\{ \left(\frac{3}{2}\rmax + \min\{d,M\}\right)^2,(R_{\max}^2+1)  \right\} + a\left(\frac{3}{2}\rmax + \min\{d,M\}\right)\right) \\
\le & 2M\left(2\rmax + \min\{d,M\}\right)^2\exp\left(-\left(\half{\rmin}-a\right)^2\sqrt{\min\{d,M\}}/8\right)
}
The max will always be achieved at the first term as $\min\{d,M\}\ge 1$.
Similarly, with the same rotation, for $U_i,i\ne 1$, 
\bas{
U_i \le \sup_u \|\bE w_1(X;\bmu^u)w_i(X;\bmu^u)(X-\bmu_1^u)(X-\bmu_i^u)^T\|_{op}+a\|\bE w_1(X;\bmu^u)w_i(X;\bmu^u)(X-\bmu_i^u)\|
}
By Lemma~\ref{lem:wiwj_bound}, when $\rmin = \Omega(\sqrt{\min\{d,M\}})$, we have
\bas{
U_i \le & \exp\left(-\left(\half{\rmin}-a\right)^2\sqrt{\min\{d,M\}}/8\right)\cdot \\
&   \left(\max\left\{ (1+2\kappa)\left(\half{3}\rmax+\min\{d,M\}\right)^2  , 2M(R_{\max}^2+1) \right\}
+ 2Ma\left(\frac{3}{2}\rmax + \min\{d,M\}\right)\right) \\
\le &  \exp\left(-\left(\half{\rmin}-a\right)^2\sqrt{\min\{d,M\}}/8\right)\left( \frac{3}{2}\rmax+ \min\{d,M\} \right) \\
& \qquad \qquad \cdot \left( \max\{(1+2\kappa),2M\}\left( \frac{3}{2}\rmax+ \min\{d,M\} \right)+2Ma \right)\\
\le &  \exp\left(-\left(\half{\rmin}-a\right)^2\sqrt{\min\{d,M\}}/8\right)\left( \frac{3}{2}\rmax+ \min\{d,M\} \right)^2\cdot \max\{3M,M+2\kappa+1\}\\
\le & M(2\kappa+4)\left( \frac{3}{2}\rmax+ \min\{d,M\} \right)^2\exp\left(-\left(\half{\rmin}-a\right)^2\sqrt{\min\{d,M\}}/8\right)
}
The second inequality is because $R_{\max}^2+1 \le\left(\half{3}\rmax+\min\{d,M\}\right)^2$ and the third inequality is because $2a \le \half{3}\rmax+\min\{d,M\}$. 
Taking back to Eq.~\eqref{eq:u1_plus_ui}, and summing over $i\in [M]$, we have
\bas{
&\|\nabla_{\bmu_i} Q(\bmu|\bmu^t) - \nabla_{\bmu_i} q(\bmu)\|\\ 
\le& M(2\kappa+4)\left( 2\rmax+ \min\{d,M\} \right)^2\exp\left(-\left(\half{\rmin}-a\right)^2\sqrt{\min\{d,M\}}/8\right)\sum_{i=1}^M\|\bmu^t_i-\bmu_i^*\|
}

This completes the proof.
\end{proof}

\subsection{Proof of Theorem \ref{th:population_contraction}}
\label{sec:proof_main_popu}

\begin{proof}[Proof of Theorem \ref{th:population_contraction}]
By Theorem \ref{th:gmm_gs} and Theorem \ref{th:contraction_with_gamma}, it suffices to check $\gamma \le \pi_{\min}$. Solving the inequality we have 
\bas{
a\le \half{\rmin} - \frac{2\sqrt{2}}{\sqrt[4]{\min\{d,M\}}} \sqrt{\log\left( \frac{M^2(2\kappa+4)(2\rmax+\min\{d,M\})^2}{\pi_{\min}} \right)}
}
Combined with the condition in Theorem \ref{th:gmm_gs}, we have
\bas{
a\le& \half{\rmin}-\max\left\{ \frac{2\sqrt{2}}{\sqrt[4]{\min\{d,M\}}}\sqrt{\log\left(\frac{M^2(2\kappa+4)(2\rmax+\min\{d,M\})^2}{\pi_{\min}}\right)}, \right. \\
& \qquad\qquad\qquad  \left. \sqrt{\min\{d,M\}}\max(4\sqrt{2[\log (\rmin/4)]_+},8\sqrt{3}) \right\} \\
=& \half{\rmin}- \sqrt{\min\{d,M\}} o(\rmin)
}
because
\bas{
& \max \left\{ c\sqrt{\log( c_1\frac{M^2\kappa}{\pi_{\min}} +2\log \left(2\rmax+\min\{d,M\}\right)}, \sqrt{\min\{d,M\}}\max\{c_2\sqrt{\log (\rmin/4)_+},8\sqrt{3}\} \right\} \\
\leq & \max \left\{ c \sqrt{\log(c_1\frac{M^2\kappa}{\pi_{\min}}+c_2\rmax+c_3{\min\{d,M\}})}, c'\sqrt{\min\{d,M\}}\sqrt{\log(\rmax+e)} \right\}\\
\leq & \sqrt{\min\{d,M\}}O\left(\sqrt{\log\left(\max\left\{\frac{M^2\kappa}{\pi_{\min}},\rmax,\min\{d,M\}\right\}\right)}\right)
}
The condition in Theorem \ref{th:gmm_gs} can be rewritten as 
\bas{
a \leq \half{\rmin} - \sqrt{\min\{d,M\}}O\left(\sqrt{\log\left(\max\left\{\frac{M^2\kappa}{\pi_{\min}},\rmax,\min\{d,M\}\right\}\right)}\right)
}

\end{proof}

\section{Proofs for sample-based gradient EM}
\label{app:proof_sample}
In this section we develop the error bound for sample-based gradient EM. Our proof is based on the Rademacher complexity theory and some new tools for contraction result. 
In \cite{maurer2016vector}, Maurer has the following contraction result for the complexity defined over countable sets.
\begin{lemma}[Theorem 3 \cite{maurer2016vector}]
	\label{lem:maurer}
Let $X$ be nontrivial, symmetric and sub-gaussian. Then there exists a constant $C<\infty$, depending only on the distribution of $X$, such that for any countable set $\cS$ and function $h_i:\cS \to \bR$, $f_i: \cS\to \bR^k$, $i\in [n]$ satisfying
$ \forall s,s'\in \cS, |h_i(s)-h_i(s')|\le L \|f(s)-f(s')\|$. If $\epsilon_{ik}$ is an independent doubly indexed Rademacher sequence, 
we have,
\bas{
\bE \sup_{s\in \cS}\sum_i \epsilon_i h_i(s)\le \bE \sqrt{2}L \sup_{s\in \cS}\sum_{i,k}\epsilon_{ik}f_i(s)_k,
}
where $f_i(s)_k$ is the $k$-th component of $f_i(s)$.
\label{lem:vector_contraction}
\end{lemma}
We prove Lemma~\ref{lem:vector_contraction_separable} by generalizing this result to any subset of separable Banach space.

\begin{proof}[Proof of Lemma~\ref{lem:vector_contraction_separable}]
	First note that a subset of a separable subspace is separable, and has a dense countable subset; lets call this $\cS_0$.
	Now note that if the Lipschitz condition holds for $s,s'\in \cS$, then it also holds for $s,s'\in \cS_0$.
	Now applying Lemma~\ref{lem:maurer}, we see that 
	\bas{
		\bE \sup_{s\in \cS_0}\sum_i \epsilon_i h_i(s)\le \bE \sqrt{2}L \sup_{s\in \cS_0}\sum_{i,k}\epsilon_{ik}f_i(s)_k,
	}
All we need to prove is that the two supremas over $\cS_0$ on the LHS and RHS of the above equation can be replaced by supremum over $\cS$. We will only show this for the LHS. The argument for the RHS is identical. In order to show this, we need to also make sure that $g(s):=\sum_i \epsilon_ih_i(s)$ over $\cS$ is measurable. We show this using standard tools from measure theory.  

We want to show that:
\ba{\label{eq:sup-countable}
\sup_{s\in\cS}g(s)=\sup_{s\in\cS_0}g(s).
}
Since $g(s)$ is continuous, its also measurable for all $s\in \cS$. The above statement, once proven, essentially implies that the sup over $\cS$ is the same as the sup over a countable set $\cS_0$. Since pointwise sup over measurable functions is measurable, we are done. We now prove Eq.~\eqref{eq:sup-countable}. It is clear that,
$\sup_{s\in\cS}g(s)\geq \sup_{s\in\cS_0}g(s).$
So all we need is to prove that for all $\epsilon>0$.
\ba{\label{eq:reverse}
\sup_{s\in\cS}g(s)\leq\sup_{s\in\cS_0}g(s)+\epsilon}
Since $g(s)$ is continuous, let $D_1(s)=\{s'\in \cS:|g(s)-g(s')|\leq \epsilon\}$. Furthermore, since $\cS_0$ is dense in $\cS$, we also have $D_2(s,\epsilon):=D_1(s)\cap \cS_0\neq \phi$. So for each $s\in \cS$, and $\epsilon>0$, $\exists s'\in \cS_0$ (to be precise, $s'\in D_2(s,\epsilon)$)
such that $g(s)\leq g(s')+\epsilon$. Taking a sup over the LHS over $\cS$ and a sup of RHS over $\cS_0$, we get Eq.~\eqref{eq:reverse}. This completes the proof.
\end{proof}

\begin{proof}[Proof of Proposition~\ref{prop:rademacher_bound}]
For any unit vector $u$, the Rademacher complexity of $\cF$ is
\beq{
\bsplt{
R_n(\cF) =& \bE_X \bE_\epsilon \sup_{\bmu\in \bA} \frac{1}{n}\sum_{i=1}^n \epsilon_i w_1(X_i;\bmu)\ip{X_i-\bmu_1,u}\\
\le & \underbrace{\bE_X \bE_\epsilon \sup_{\bmu\in \bA}\frac{1}{n}\sum_{i=1}^n \epsilon_i w_1(X_i;\bmu)\ip{X_i,u}}_{(D)} +  \underbrace{\bE_X \bE_\epsilon \sup_{\bmu\in \bA} \frac{1}{n}\sum_{i=1}^n \epsilon_iw_1(X_i;\bmu)\ip{\bmu_1,u}}_{(E)}
}
\label{eq:mgf_cauchy_schwarz}
}
We bound the two terms separately. 
Define $\eta_j(\bmu): \mathbb{R}^{Md}\to \mathbb{R}^{M}$ to be a vector valued function with the $k$-th coordinate  
\bas{
[\eta_j(\bmu)]_k=\frac{\norm{\bmu_1}^2}{2}-\frac{\norm{\bmu_k}^2}{2}+\langle X_j,\bmu_k-\bmu_1\rangle+\log\left(\frac{\pi_k}{\pi_1}\right)
}
We claim
\beq{
|w_1(X_j;\bmu)-w_1(X_j;\bmu')|\leq \frac{\sqrt{M}}{4}\norm{\eta_j(\bmu)-\eta_j(\bmu')}
\label{eq:lipschitz}
}
This vectorized Lipschitz condition simply follows from the fact that 
\bas{
w_1(X_j,\bmu)= &\frac{1}{1+\sum_{k=2}^{M}\exp([\eta_j(\bmu)]_k)}\\
\frac{\partial w_1(X_j,\bmu)}{\partial [\eta_j(\bmu)]_k}= &\frac{\exp([\eta_j(\bmu)]_k)}{(1+\sum_{k=2}^{M}\exp([\eta_j(\bmu)]_k))^2} \leq \frac{1}{4}
}
so $w_1(X_j,\bmu)$ is $\frac{1}{4}$-Lipschitz continuous w.r.t. $[\eta_j(\bmu)]_k$. By Lemma \ref{lem:vec_Lipschitz}, $w_1(X_j,\bmu)$ is $\frac{\sqrt{M}}{4}$ Lipschitz w.r.t $\eta_j(\bmu)$.
Now let $\psi_j(\bmu)=w_1(X_j;\bmu)\langle X_j,u \rangle$. 

With Lipschitz property \eqref{eq:lipschitz} and by Lemma~\ref{lem:vector_contraction}, 
we have
\beq{
\bsplt{
&\bE \left[{\sup_{\bmu \in \mathbb{A}}\frac{1}{n}\sum_{j=1}^n\epsilon_j w_1(X_j;\bmu)\langle X_j,u \rangle}\right] \le 
\bE \left[ \frac{1}{n}\sup_{\bmu \in \mathbb{A}}\sum_{j=1}^n\sum_{k=1}^M \epsilon_{jk}[\eta_j(\bmu)]_k \frac{\sqrt{2M}}{4} \langle X_j,u \rangle \right]\\
= & \bE \left[\frac{\sqrt{2}M^{\frac{1}{2}}}{4n}\sup_{\bmu \in \mathbb{A}}\sum_{j=1}^n\sum_{k=2}^M \epsilon_{jk} \left( \frac{\norm{\bmu_1}^2}{2}-\frac{\norm{\bmu_k}^2}{2}+\langle X_j,\bmu_k-\bmu_1\rangle+ \log(\frac{\pi_k}{\pi_1}) \right) \langle X_j,u \rangle \right]\\
\leq & \underbrace{\bE \left[ \frac{\sqrt{2M}}{4n} \sup_{\bmu \in \mathbb{A}}\sum_{j=1}^n\sum_{k=1}^M \epsilon_{jk} \left( \frac{\norm{\bmu_1}^2}{2}-\frac{\norm{\bmu_k}^2}{2}+  \log(\frac{\pi_k}{\pi_1}) \right) \langle X_j,u \rangle \right]}_{(D.1)}\\
&\qquad\qquad\qquad\qquad +\underbrace{\bE \left[ \frac{\sqrt{2M}}{4n} \sup_{\bmu \in \mathbb{A}}\sum_{j=1}^n\sum_{k=1}^M \epsilon_{jk}  \langle X_j,\bmu_k-\bmu_1\rangle \langle X_j,u \rangle  \right]}_{(D.2)}
}
\label{eq:EewtimesX}
}
To bound (D.1), note that the sum over $k=1,\cdots,M$ can be considered as an inner product of two vectors in $\bR^M$. The supremum of $\|\bmu\|$ can be bounded as $\max_{\bmu\in \bA} \|\bmu_i\| \le \|\bmu^*_{\max}\|+a \le \frac{3}{2}\rmax$.
\ba{
(D.1) =& \bE \left[ \frac{\sqrt{2M}}{4} \sup_{\bmu \in \mathbb{A}} 
\begin{pmatrix} 
\frac{\norm{\bmu_1}^2}{2}-\frac{\norm{\bmu_1}^2}{2}+\log(\frac{\pi_1}{\pi_1})\\
\vdots\\
\frac{\norm{\bmu_1}^2}{2}-\frac{\norm{\bmu_M}^2}{2}+\log(\frac{\pi_M}{\pi_1})
\end{pmatrix}^T
\begin{pmatrix}
\frac{1}{n}\sum_{j=1}^n \epsilon_{j1} \ip{X_j,u } \\
\vdots\\
\frac{1}{n}\sum_{j=1}^n \epsilon_{jM} \ip{X_j,u } 
\end{pmatrix} \right]\nonumber\\
\le & cM(9\rmax^2/4+\log(\kappa))\bE \left\| \begin{pmatrix}
\frac{1}{n}\sum_{j=1}^n \epsilon_{j1} \ip{X_j,u } \\
\vdots\\
\frac{1}{n}\sum_{j=1}^n \epsilon_{jM} \ip{X_j,u } 
\end{pmatrix} \right\|\label{eq:D.1}
}
By Lemma~\ref{lem:mixture_sub_gaussian}, and $\|u\|= 1$, we know $\langle X_j,u \rangle$ is sub-Gaussian with parameter upper bounded by $1+\rmax$. So each element of the vector in Equation~\ref{eq:D.1} is the average of $n$ independent mean 0 sub-Gaussian random variables with sub-gaussian norm upper bounded by $1+\rmax$ (since w.l.o.g we have assumed that $\sigma=1$ and $\max_i\|\mu\|\leq \rmax$, by Lemma~\ref{lem:rmax_mumax}). 
Consequently, $\forall k\in[M]$, $\bE\left| \frac{1}{n}\sum_{j=1}^n \epsilon_{jk} \ip{X_j,u_1}\right| \le c(1+\rmax)/\sqrt{n}$ for some global constant $c$~\cite{vershynin2010introduction}, and
\bas{
(D.1) \leq cM^{3/2}(9\rmax^2/4+\log(\kappa))(1+\rmax)\frac{1}{\sqrt{n}} \le cM^{3/2}(1+\rmax)^3\max\{1,\log(\kappa)\}\frac{1}{\sqrt{n}}
}
On the other hand, for $(D.2)$, we have
\beq{
\bsplt{
(D.2) =& \bE \left[ \frac{\sqrt{2M}}{4n} \sup_{\bmu \in \mathbb{A}}\sum_{j=1}^n\sum_{k=1}^M \epsilon_{jk}  \langle X_j,\bmu_k-\bmu_1\rangle \langle X_j,u \rangle  \right]\\
&= \bE \left[ \frac{\sqrt{2M}}{4n} \sup_{\bmu \in \mathbb{A}}\sum_{k=1}^M (\bmu_k-\bmu_1)^T\left(\sum_{j=1}^n \epsilon_{jk}  X_jX_j^T\right)u  \right]\\
\le & \sum_{k=1}^M \bE \left[ \frac{\sqrt{2M}}{4} \sup_{\bmu \in \mathbb{A}} \|\bmu_k-\bmu_1\| \left\|\frac{1}{n}\sum_{j=1}^n \epsilon_{jk}  X_j X_j^T \right\|_{op}  \right] \\
\le&\sum_{k=1}^M  \frac{\sqrt{2M}}{2} \|\bmu_{\max}\| \bE \left[ \left\|\frac{1}{n}\sum_{j=1}^n \epsilon_{jk}  X_j X_j^T\right\|_{op}  \right] 
}
\label{eq:d_2}
}
For each $k\in[M]$, the operator norm $\|\frac{1}{n}\sum_{j=1}^n \epsilon_{jk}  X_jX_j^T\|_{op}$ can be bounded by the same discretization technique with the $1/2$-covering of the unit sphere. To be specific, since for any matrix $A$,
$\|A\|_{op} = \sup_{u\in \cS^{d-1}} \|Au\|$,
\bas{
\forall u, \exists u_j\ s.t.\  &\|Au\| \le \|Au_j\|+\|A\|_{op}\|u-u_j\|\le \max_j \|Au_j\|+\frac{1}{2}\|A\|_{op}
}
Taking $\sup_{u\in \cS^{d-1}}$ on the left side, we get $\|A\|_{op} \le 2\max_j \|Au_j\|.$
Therefore $\|\frac{1}{n}\sum_{j=1}^n\epsilon_{jk}  X_jX_j^T\|_{op} \le 2\max_\ell \frac{1}{n}\sum_{j=1}^n\epsilon_{jk}  \ip{X_j,u_\ell}^2$.
The square of sub-gaussian random variable $\ip{X_j,u_\ell}$ is sub-exponential, from Lemma 5.14 in \cite{vershynin2010introduction} we know
\bas{
\bE \left[\exp\left( \frac{1}{n}\sum_{j=1}^n \epsilon_{jk} \ip{X_j,u} ^2 t\right) \right] \le & \exp\left( \frac{c_4t^2 (1+\rmax)^4}{n} \right)
}
With the 1/2-covering number of $\cS^{d-1}$ bounded by $\exp(2d)$, we have
\bas{
\bE \left[\exp\left( t\cdot \|\frac{1}{n}\sum_{j=1}^n \epsilon_{jk}  X_jX_j^T\|_{op}\right) \right] \le & \exp\left( 2d+\frac{c_5t^2 (1+\rmax)^4}{n} \right)
}
Hence,
\bas{
\bE \left[ \left\|\frac{1}{n}\sum_{j=1}^n \epsilon_{jk}  X_j X_j^T\right\|_{op}  \right] = & \frac{1}{t}\log\left( \exp\left( t\bE  \left[ \left\|\frac{1}{n}\sum_{j=1}^n \epsilon_{jk}  X_j X_j^T\right\|_{op}  \right]\right) \right), \quad \forall t>0\\
\le& \frac{1}{t}\log\left( \bE  \left[ \exp \left( t  \left\|\frac{1}{n}\sum_{j=1}^n \epsilon_{jk}  X_j X_j^T\right\|_{op}  \right)\right] \right)\\
\le & \frac{2d}{t}+\frac{ct(1+\rmax)^4}{n}
}
Taking $t=c\frac{\sqrt{nd}}{(1+\rmax)^2}$,
\bas{
\bE \left[ \left\|\frac{1}{n}\sum_{j=1}^n \epsilon_{jk}  X_j X_j^T\right\|_{op}  \right] \le c\sqrt{\frac{d}{n}}(1+\rmax)^2
}
Plugging back to Eq.~\eqref{eq:d_2}, and use $\sup_{\bmu\in \bA}\|\bmu\| \le \sup_k \|\bmu^*_k\|+a \le \frac{3}{2}\rmax$, we have
\bas{
(D.2) \le & \frac{cM(1+\rmax)^3\sqrt{d}}{\sqrt{n}}
}

Plugging the bound back to Eq.~\eqref{eq:EewtimesX}, we have
\bas{
(D) \le&\frac{c M^{3/2}(1+\rmax)^3\sqrt{d}\max\{1,\log(\kappa)\}}{\sqrt{n}}
}

Apply Lemma~\ref{lem:vector_contraction} on the $(E)$ term in Eq.~\eqref{eq:mgf_cauchy_schwarz}, we have
\bas{
(E)=&\bE \left[{\sup_{\bmu \in \mathbb{A}}\frac{1}{n}\sum_{j=1}^n\epsilon_j w_i(X_j;\bmu)\langle \bmu_i,u\rangle}\right] \\
\leq& \bE\left[\frac{\sqrt{2M}}{4n}\sup_{\bmu \in \mathbb{A}}\sum_{j=1}^n\sum_{k=1}^M \epsilon_{jk} \left( \frac{\norm{\bmu_1}^2}{2}-\frac{\norm{\bmu_k}^2}{2}+\langle X_j,\bmu_k-\bmu_1\rangle+\log(\frac{\pi_k}{\pi_1}) \right) \langle \bmu_i,u \rangle
\right]\\
\le & \underbrace{\frac{\sqrt{2M}}{4}\bE_\epsilon\left[ \sup_{\bmu\in \bA} \frac{1}{n}\sum_{j=1}^n\sum_{k=1}^M \epsilon_{jk} \left( \frac{\norm{\bmu_1}^2}{2}-\frac{\norm{\bmu_k}^2}{2}+\log\frac{\pi_k}{\pi_1}\right) \ip{\bmu_i,u} \right]}_{E.1} \\
&\qquad\qquad\qquad + \underbrace{\frac{\sqrt{2M}}{4}\bE_{X,\epsilon}\left[ \sup_{\bmu\in \bA}\frac{1}{n}\sum_{j=1}^n\sum_{k=1}^M \epsilon_{jk} \ip{X_j,\bmu_k-\bmu_1}\ip{\bmu_i, u }\right]}_{E.2}
}
We will now bound $(E.1)$ and $(E.2)$.
\ba{
(E.1) &\leq \frac{\sqrt{2M}}{4}\bE_\epsilon\left[ \sup_{\bmu\in \bA} \frac{1}{n}\sum_{j=1}^n\sum_{k=1}^M \epsilon_{jk} \left( \frac{\norm{\bmu_1}^2}{2}-\frac{\norm{\bmu_k}^2}{2}+\log\frac{\pi_k}{\pi_1} \right) \sup_{\bmu\in \bA}\ip{\bmu_i,u} \right]\nonumber\\
&\leq \frac{\sqrt{2M}}{4}\rmax\bE_\epsilon\left[ \sup_{\bmu\in \bA} \begin{pmatrix} 
	\frac{\norm{\bmu_1}^2}{2}-\frac{\norm{\bmu_1}^2}{2}+\log(\frac{\pi_1}{\pi_1})\\
	\vdots\\
	\frac{\norm{\bmu_1}^2}{2}-\frac{\norm{\bmu_M}^2}{2}+\log(\frac{\pi_M}{\pi_1})
\end{pmatrix}^T
\begin{pmatrix}
	\frac{1}{n}\sum_{j=1}^n \epsilon_{j1}  \\
	\vdots\\
	\frac{1}{n}\sum_{j=1}^n \epsilon_{jM}  
\end{pmatrix} \right]\nonumber\\
&\leq cM\rmax(9\rmax^2/4+\log\kappa)E_\epsilon\left\|\begin{pmatrix}
	\frac{1}{n}\sum_{j=1}^n \epsilon_{j1}  \\
	\vdots\\
	\frac{1}{n}\sum_{j=1}^n \epsilon_{jM}  
\end{pmatrix} \right\|\label{eq:E.1}
}
Note that each element of the vector in Equation~\ref{eq:E.1} is the average of $n$ i.i.d mean 0 Radamacher random variables, which are essentially sub-gaussian radnom variables with subgaussian norm upper bounded by $1$. 
Consequently, $\forall k\in[M]$, $\bE\left| \frac{1}{n}\sum_{j=1}^n \epsilon_{jk} \right| \le c'/\sqrt{n}$ for some global constant $c$~\cite{vershynin2010introduction}, and
\bas{
	(E.1)\leq c'M^{3/2}\rmax(9\rmax^2/4+\log\kappa)/\sqrt{n}
}
As for (E.2), we have
\ba{
(E.2)&\leq \frac{\sqrt{2M}}{4}\bE_{X,\epsilon}\left[ \sup_{\bmu\in \bA} \frac{1}{n}\sum_{j=1}^n\sum_{k=1}^M \epsilon_{jk} \langle X_j,\bmu_k-\bmu_1\rangle \sup_{\bmu\in \bA}\ip{\bmu_i,u} \right]\nonumber\\
&\leq \frac{3\sqrt{2M}}{8}\rmax \bE_{X,\epsilon}\left[ \sup_{\bmu\in \bA} \sum_{k=1}^M  (\bmu_k-\bmu_1)^T\left(\frac{1}{n}\sum_{j=1}^n \epsilon_{jk}X_j\right) \right]\nonumber\\
&\leq \frac{3\sqrt{2M}}{8}\rmax  \sum_{k=1}^M\bE_{X,\epsilon}\left[  \sup_{\bmu\in \bA} (\bmu_k-\bmu_1)^T\left(\frac{1}{n}\sum_{j=1}^n \epsilon_{jk}X_j \right)\right]\nonumber\\
&\leq \frac{9\sqrt{2M}}{8}\rmax^2 \sum_{k=1}^M\bE_{X,\epsilon} \left\|\frac{1}{n}\sum_{j=1}^n \epsilon_{jk}X_j \right\|\label{eq:E.2}
}
In Eq~\eqref{eq:E.2}, the vector $\frac{1}{n}\sum_{j=1}^n \epsilon_{jk}X_j$ is the average of $n$ independent mean zero isotropic subgaussian random vectors. Another using of the discretizing technique along with the moment generating function with $t\geq 0$ gives:
\bas{
\left\|\frac{1}{n}\sum_{j=1}^n \epsilon_{jk}X_j \right\|&\leq 2\max_\ell\langle\frac{1}{n}\sum_{j=1}^n \epsilon_{jk}X_j,u_\ell\rangle	\\
E\left[\exp \left( t\left\|\frac{1}{n}\sum_{j=1}^n \epsilon_{jk}X_j \right\| \right)\right]
&\leq \sum_{\ell}E\left[\exp \left(2\frac{t}{n}\sum_{j=1}^n \epsilon_{jk}\langle X_j,u_\ell\rangle\right) \right]\leq  \exp\left(2d+\frac{c'(1+\rmax)^2t^2}{n}\right)\\
 E\left\|\frac{1}{n}\sum_{j=1}^n \epsilon_{jk}X_j \right\|&\leq  \frac{c''+2d+\frac{c'(1+\rmax)^2t^2}{n}}{t} \qquad \text{Using Jensen's inequality}
}
 Taking $t=\Theta \left(\sqrt{nd}/(1+\rmax) \right)$,
 \bas{
(E.2) \leq c M^{3/2}\rmax^2 (1+\rmax)\sqrt{d}/\sqrt{n}
}
Thus, combing (E.1) and (E.2) we get:
\bas{
(E)\le &  \frac{c M^{3/2}(1+\rmax)^3\max\{1,\log(\kappa)\}\sqrt{d}}{\sqrt{n}}
}
The final bound follows by combining (D) and (E):
\bas{
R_n(\cF) \le \frac{c M^{3/2}(1+\rmax)^3\sqrt{d}\max\{1,\log(\kappa)\}}{\sqrt{n}}
}
\end{proof}
For proving Lemma~\ref{lem:zuj_minus_ezuj} we first recall the following symmetrization lemma in learning theory.
\begin{lemma}[See e.g. \cite{mohri2012foundations}]
Let $\mathcal{F}$ be a function class with domain X. Let $\{X_1,X_2,\cdots,X_n\}$ be a set of sample generated by a distribution $\mathbb{P}$ on X. Assume $\sigma_i$ are $i.i.d.$ Rademacher variables, then 
\bas{
\bE \left( \sup_{f \in \mathcal{F}}(\bE f-\frac{1}{n}\sum_{i=1}^{n}f(X_i))\right) \leq 2R_n(\mathcal{F})
}
Here $R_n(\mathcal{F})=\bE\left[\sup_{f \in \mathcal{F}}|\frac{1}{n}\sum_{i=1}^{n} \sigma_if(X_i)\right]$ is the Rademacher complexity.
\label{lem:symmetrization}
\end{lemma}

\begin{proof}[Proof of Lemma \ref{lem:zuj_minus_ezuj}]
Consider for some $r>0$, the set in which $X$ lies in the $r$-ball of its corresponding center. If $Z_{i}$ denotes the hidden cluster assignment of $X_i$, we denote $\cY_r^i=\{X_1,\cdots , X_n: \|X_i-\bmu^*_{Z_{i}}\| \le r\}$.
\bas{
\cY_r := \{X_1,\dots X_n:\|X_i-\bmu^*_{Z_{i}}\| \le r,\ , \forall i\in[n] \}=\cap_i \cY_r^i
}
By Lemma \ref{lem:norm_of_gaussian} and union bound,  for $r=\Omega(\sqrt{d})$, 
\ba{\label{eq:badprob}
	p:=P(\bX \not\in \cY_r)\le \sum_{i=1}^n P(X\in (\cY_r^{i})^c )\le cn\exp\left(-\half{r\sqrt{d}}\right). 
}

Let $m_r := \bE[g(X)|X\in \cY_r]$, we want to show $m_r$ is close to $\bE[g(\bX)]$ and is close to $g(\bX)$ with high probability. 

Let $\bX$ and $\bX'$ be two samples which only differ on one data-point, then 
\bas{
g(\bX)-g(\bX') =& \sup_{\bmu\in \bA} \left(\frac{1}{n}\sum_{i=1}^n w_1(X_i;\bmu)\ip{X_i-\bmu_1,u}-\bE_X w_1(X;\bmu)\ip{X-\bmu_1,u}\right) \\
&\qquad \qquad - \sup_{\bmu\in \bA} \left(\frac{1}{n}\sum_{i=1}^n w_1(X'_i;\bmu)\ip{X'_i-\bmu_1,u}-\bE_X w_1(X';\bmu)\ip{X'-\bmu_1,u}\right)
}
Assume $\tilde{\bmu}$ be the maximizer for the supremum of $X$, then
\bas{
g(\bX)-g(\bX') 
\stackrel{(i)}{\le} & \frac{1}{n}(\sum_{i=1}^n w_1(X_i;\tilde{\bmu})\ip{X_i-\tilde{\bmu}_1,u}-\bE w_1(X;\tilde{\bmu})\ip{X-\tilde{\bmu}_1,u}) \\
&\qquad \qquad -  \frac{1}{n}\sum_{i=1}^n (w_1(X'_i;\tilde{\bmu})\ip{X'_i-\tilde{\bmu}_1,u}-\bE w_1(X';\tilde{\bmu})\ip{X'-\tilde{\bmu}_1,u})\\
= & \frac{1}{n}w_1(X_i;\tilde{\bmu})\ip{X_i-\tilde{\bmu}_1,u}-w_1(X'_i;\tilde{\bmu})\ip{X'_i-\tilde{\bmu}_1,u}
}
where (i) is by definition of supremum.
The inequality holds when we change the order of $X$ and $X'$, hence for $X,X'\in \cY_r$,
\bas{
|g(\bX)-g(\bX')| \le& \frac{1}{n}|w_1(X_i;\tilde{\bmu})\ip{X_i-\tilde{\bmu}_1,u}-w_1(X'_i;\tilde{\bmu})\ip{X'_i-\tilde{\bmu}_1,u}|\\
\le & \frac{2}{n} \sup_{\bmu\in \bA, X\in \cY_r}|w_1(X_i;\bmu)\ip{X_i-\bmu_1,u}|\\
\le & \frac{2}{n}\sup_{X\in \cY_r}(\|X-\bmu^*_{Z_X}\|+\rmax)\\
\le & \frac{2(r+\rmax)}{n}:=L
}

By Theorem \ref{th:extension_mcdiarmid}, we have
\ba{
P(g(\bX)-m_r \ge \epsilon) 
\le & p+\exp\left( -2\frac{(\epsilon-nLp)_+^2}{nL^2} \right)\nonumber\\
\le& c_1n\exp(-c\sqrt{d}r)+\exp\left( -2\frac{(\epsilon-nL\cdot c_1n\exp(-c\sqrt{d}r))_+^2}{nL^2} \right)\nonumber\\
=&\underbrace{cn\exp(-c\sqrt{d}r)}_{P_1}+\underbrace{\exp\left( -\frac{c_1n(\epsilon-c_2n(r+\rmax)\exp(-c\sqrt{d}r))_+^2}{(r+\rmax)^2} \right)}_{P_2}\label{eq:Mcdiarmid}
}
Let $r=\Theta((1+\rmax)\log^2(n)\sqrt{d})$ and $\epsilon=c_0 (1+\rmax)d\log^{5/2}(n)/\sqrt{n}$.
Since, for large $n$,
\bas{
	n(r+\rmax)\exp(-cr\sqrt{d})
	&\leq c_2(1+\rmax)\exp(\log n+\log\log n-c\log^2 n)= o((1+\rmax)/\sqrt{n})
}
for some constant $c_0$, which yields for large $n$, $(\epsilon-c_2n(r+\rmax)\exp(-cr\sqrt{d}))_+\geq \epsilon/2$. Finally, for large $n$, we can have the following bounds on $P_1$ and $P_2$.
\ba{
P_1 &= O\left(\exp(\log n-c(1+\rmax)^2(\log n)^2)\right)=O\left(\exp(-c'(1+\rmax)^2d\log n)\right)\nonumber\\
P_2 &\leq \exp\left(-\frac{cn\epsilon^2}{d(\log^2 n(1+\rmax))^2}\right) = O\left(\exp(-c''d \log n)\right)\label{eq:P12}
}
where $c,c',c'''$ are some global constants. The last line uses the fact $r+\rmax = O(\sqrt{d}(1+\rmax)\log^2 n)$.

Now we bound the difference between $\bE g(X)$ and the conditional expectation $m_r$.  
By the total expectation theorem,
\beq{
\bsplt{
&\bE g(\bX) = m_r P(\bX\in \cY_r)+\bE[g(\bX)1(\bX\not\in \cY_r)]\\
& \bE [g(\bX)] ( P(\bX\in \cY_r) + P(\bX\not\in \cY_r)) = m_r P(\bX\in \cY_r)+\bE[g(\bX)1(\bX\not\in \cY_r)]\\
& \bE g(\bX)-m_r = \frac{ \bE[g(\bX)1(\bX\not\in\cY_r)]-\bE[g(\bX)]P(X\not\in \cY_r) }{P(\bX\in \cY_r)}\\
\Rightarrow &|m_r-\bE g(\bX)|\le \frac{p |\bE g(\bX)| +|\bE[g(\bX)1(\bX\not\in \cY_r)]| }{1-p}
}
\label{eq:conditional_exp}
}
$p$ is defined in Eq~\eqref{eq:badprob}. Note that by Proposition \ref{prop:rademacher_bound}, and the symmetrization result Lemma~\ref{lem:symmetrization}, $\bE g(X) \le 2R_n(\cF) \le c n^{-1/2}M^3\sqrt{d}(1+\rmax)^3\max\{1,\log(\kappa)\}$. On the other hand, as $g(\bX)$ is the sup over a class of quantity, which is centered at zero. So $g(X)\ge 0$. We also have $1(\bX\in \cup_i (\cY_r^i)^c)\le \sum_{i=1}^n 1(\bX\in (\cY_r^i)^c)$. Hence,
\bas{
&\bE[g(\bX)1(\bX\not\in \cY_r)] = \bE[g(\bX)  \sum_{i=1}^n  1(\bX\in (\cY_r^i)^c)] \le  \sum_{i=1}^n   \bE[g(\bX)1(\bX\in (\cY_r^i)^c)] }

Note for each sample $X_i$ and $\bmu$, $\left| \sup_{\bmu\in \bA} w_1(X_i; \bmu)\ip{X_i- \bmu_1,u} \right| \le \sup_{\bmu\in \bA} w_1(X_i; \bmu) \|X_i-\bmu^*_{Z_{i}}\|+\|\bmu^*_{Z_{i}}-\bmu_1\| \le \|X_i-\bmu^*_{Z_{i}}\| + 2\rmax$. Thus,
\bas{
|g(\bX)| =& | \sup_{\bmu\in \bA} \frac{1}{n}\sum_{j=1}^n w_1(X_j;\bmu)\ip{X_j-\bmu_1,u}-\bE_X w_1(X;\bmu)\ip{X-\bmu_1,u}|\\
\le &  \frac{1}{n}\sum_{j=1}^n (\|X_j-\bmu^*_{Z_{j}}\|+2\rmax) + \bE_X \|X - \bmu^*_{Z_X}\|+2\rmax \\
\le & \frac{1}{n} \sum_{j=1}^n \|X_j-\bmu^*_{Z_{j}}\| + \bE_X \|X - \bmu^*_{Z_X}\| +4\rmax
}
Therefore we have,
\beq{
\bsplt{
&\bE[g(\bX)1(\bX\not\in \cY_r)] \le  \sum_{i=1}^n   \bE_{\bX}[ ( \frac{1}{n} \sum_{j=1}^n \|X_j-\bmu^*_{Z_{j}}\| + \bE_X \|X - \bmu^*_{Z_X}\| + 4\rmax) 1(\bX\in (\cY_r^i)^c)]\\
\le & \sum_{i=1}^n   \bE_{\bX}[ ( \frac{1}{n} \sum_{j=1}^n \|X_j-\bmu^*_{Z_{j}}\| 1(\bX\in (\cY_r^i)^c)] ] + (\bE_X \|X - \bmu^*_{Z_X}\| + 4\rmax)P(\bX\in (\cY_r^i)^c)]\\
\le &  \sum_{i=1}^n  \frac{1}{n} \sum_{j=1}^n\bE_{\bX}[\|X_j-\bmu^*_{Z_{j}}\| 1(\bX\in (\cY_r^i)^c)]+  c'(\rmax+d)p
}
\label{eq:tail_exp_for_g}
}
where the last inequality follows from Lemma~\ref{lem:normal_norm}.
Note that when $j\ne i$, the expectation factors due to independence of the sample points and by Lemma \ref{lem:norm_of_gaussian},
\bas{
\bE_{\bX}[\|X_j-\bmu^*_{Z_{j}}\| 1(\bX\in (\cY_r^i)^c)] = \bE_{X_j}\|X_j-\bmu^*_{Z_{j}}\|\cdot  P(\|X_i-\bmu^*_{Z_i}\| \ge r) \le  cde^{-\half{r\sqrt{d}}}
}
When $j=i$, from Lemma~\ref{lem:pol_exp_tail_bound},
\bas{
\bE_{\bX}[\|X_j-\bmu^*_{Z_{j}}\| 1(\bX\in (\cY_r^i)^c)]  \le&  cn\int_{v=r}^\infty v \cdot v^{d-1}(v+\rmax+a)\exp(-v^2/2)dv\cdot \frac{2\pi^{d/2}}{\gamfun{\half{d}}}\\
\le & c_1d\exp\left( -\frac{r\sqrt{d}}{2}\right)
}
Putting back to Eq.~\eqref{eq:tail_exp_for_g}, we have
\bas{
\bE[g(\bX)1(\bX\not\in \cY_r)] \le&  c_1nd\exp\left(-\half{r\sqrt{d}}\right)+c_2 d \exp\left(- \frac{r\sqrt{d}}{2}\right)+c_3n(\rmax+d)\exp\left(-\half{r\sqrt{d}}\right)\\
\le& cn(\rmax+d)\exp\left(- \frac{r\sqrt{d}}{2}\right)
}

Following from Eq.~\eqref{eq:conditional_exp}, we have
\ba{\label{eq:mr}
|m_r-\bE g(X)| \le &
 \frac{c_1n\exp(-\half{r}\sqrt{d})R_n(\cF)+c_2n(\rmax+d)\exp(-\half{r}\sqrt{d})}{1-c_3n\exp(-c_4r\sqrt{d})}
}
Recall that we take $r=\Theta(\sqrt{d}(1+\rmax)\log^2 n)$, for large enough $n$, we have $1-c_3n\exp(-cr\sqrt{d})\geq 1/2$, and $ne^{-cr\sqrt{d}}\leq C/n$.
Finally for the second part of the numerator in Eq.~\eqref{eq:mr} we have:
\bas{ 
	n(\rmax+d)\exp(-\sqrt{d}r/2)&\leq (\rmax+1)\exp(\log n+\log d-\Theta(d(1+\rmax)\log^2 n))\\
	&\le  C'(\rmax+1)/\sqrt{n}.
}
Eq~\eqref{eq:mr} becomes,
\ba{\label{eq:finalmr}
m_r\leq 2R_n(\cF)(1+O(1/n))+O((\rmax+1)/\sqrt{n})
}

Thus using Eqs~\eqref{eq:Mcdiarmid},~\eqref{eq:P12} and~\eqref{eq:finalmr} the final bound becomes:

\bas{
&P(g(X)\leq 2R_n(\cF)(1+O(1/n))+O((\rmax+1)/\sqrt{n})+(1+\rmax)d\sqrt{\log^5 n/n})\\
\ge & 1-P_1-P_2\\
\ge & 1-c\exp\left(-c'\min((1+\rmax)^2d\log n,d\log n)\right) \geq 1-\exp\left(-cd\log n\right)
}
Finally we have,
\bas{
	P(g(\bX) &= \tilde{O}(\max\{R_n(\cF), (1+\rmax)dlog^{5/2}(n)/\sqrt{n}\}) )\geq 1-\exp\left(-cd\log n\right)
}
\bk
\end{proof}

\begin{proof}[Proof of Theorem~\ref{th:sample_eps}]

Denote $Z_i=\sup_{\bmu \in \mathbb{A}}\norm{G^{(i)}(\bmu)-G^{(i)}_n(\bmu)}$ where 
\bas{
G(\bmu)=\begin{pmatrix} \bE w_1(X;\bmu)(X-\bmu_1)\\ 
\bE w_2(X;\bmu)(X-\bmu_2)\\
\vdots\\
\bE w_M(X;\bmu)(X-\bmu_3)
\end{pmatrix}.
}

Assume $\cS^{d-1}$ is $d$-dimensional unit sphere. Recall the definition $g_i^u(X) = \sup_{\bmu\in \mathbb{A}}\langle G^{(i)}(\bmu)-G^{(i)}_n(\bmu),u\rangle = \sup_{\bmu\in \bA} \frac{1}{n}\sum_{i=1}^n w_1(X_i;\bmu)\ip{X_i-\bmu_1,u}-\bE w_1(X;\bmu)\ip{X-\bmu_1,u}$. Then $Z_i=\sup_{u\in \cS^{d-1}} g_i^u(X)$. 
Without loss of generality, we assume $i=1$, the proof for other clusters follows similarly. 
Let $\{u^{(1)}, u^{(2)}, \cdots, u^{(K)}\}$ be a $\frac{1}{2}$-covering of the unit sphere $\cS^{d-1}$, then $\forall v \in \cS^{d-1}, \exists j\in[K],\text{ s.t. }\norm{v-u^{(j)}} \leq \frac{1}{2}$. Hence we have 
\bas{
g_1^v(X) \leq g_1^{u^{(j)}}(X) + |g_1^v(X) - g_1^{u^{(j)}(X)}|\leq \max_j g_1^{u^j} + Z_1 \norm{v-u^{(j)}}
}
As a result, $Z_1 \leq 2 \max_{j=1,\cdots, K} g_1^{u^{(j)}}(X)$. Therefore it is sufficient to bound $g(X)$ for a fixed $u^{(j)} \in \cS^{d-1}$. By Lemma~\ref{lem:covering_number}, covering number $K\le \exp(2d)$.

By Lemma~\ref{lem:zuj_minus_ezuj}, we have with probability at least $1-\exp\left(-cd\log n\right)$, 
$ g_1^{u_j} = \tilde{O}(\max\{ R_n(\cF_1^u), (1+\rmax)d/\sqrt{n} \})$. 
Plugging in the Rademacher complexity from Proposition \ref{prop:rademacher_bound}, and applying union bound,
we have
\bas{
Z_1 \le& 2\max_j g_1^{u_j} \le  \tilde{O}(\max\{ n^{-1/2}M^3(1+\rmax)^3\sqrt{d}\max\{1,\log(\kappa)\}, (1+\rmax)d/\sqrt{n}\})
}
with probability at least $1-\exp\left(2d-cd\log n\right)=1-\exp\left(-c'd\log n\right)$.
\end{proof}

\begin{proof}[Proof of Theorem~\ref{th:sample}]
We show 
the result by induction.
When $t=1$, 
\bas{
\norm{\bmu^{1}-\bmu^*}_2=\norm{G_n(\bmu^0)-\bmu^*}&\leq \norm{G(\bmu^0)-\bmu^*}+\norm{G_n(\bmu^0)-G(\bmu^0)}\\
&\leq \zeta\norm{\bmu^0-\bmu^*}+\epsilon^{\text{unif}}(n)
}
If $\norm{\bmu_i^t-\bmu_i^*}<a$ and $\epsilon^{\text{unif}}(n)\leq (1-\zeta)a$, we have $\norm{\bmu_i^{t+1}-\bmu_i^*}\leq a$. So $\bmu^t$ lies in the contraction region for $\forall t\ge 0$.

Then iteratively we get
\bas{
\norm{\bmu^t-\bmu^*} &\leq \zeta \norm{\bmu^{t-1}-\bmu^*}+\epsilon^{\text{unif}}(n)\\
&\leq \zeta^t \norm{\bmu^{0}-\bmu^*}+\sum_{i=0}^{t-1}\zeta^i\epsilon^{\text{unif}}(n)\\
&\leq \zeta^t \norm{\bmu^{0}-\bmu^*}+\frac{1}{1-\zeta}\epsilon^{\text{unif}}(n)
}
with probability at least $1-\delta$.
\end{proof}

\section{Improving the sample convergence rate}
\label{app:tighter_sample_bound}

We adapt Theorem~1 of  \cite{kontorovich2014concentration} to achieve the following concentration inequality.

\begin{theorem}
Let $g(X)$ be defined in Eq.~\eqref{eq:def_g} with $i=1$ and some fixed $u$, then

\bas{
P \left( g(X)-\bE g(X) > 2(1+3\rmax)\sqrt{\frac{d\log n}{n}} \right) \le & n^{ -d }
}
\label{th:extension_aryeh}
\end{theorem}

\begin{proof}
We will use the notation $X_i^j = (X_i, \cdots, X_j)$ for all sequences, and sequence concatenation is denoted multiplicatively: $x_i^jx_{j+1}^k = x_i^k$. We will also use that fact that $X_1^0$ is the empty set.
The proof will proceed via the Azuma-Hoeffding-mcDiarmid method of martingale differences. Defining $V_i = \bE[g|X_1^i]-\bE[g|X_1^{i-1}]$, we see that $g(X)-E[g(X)]=\sum_i V_i$. We also note that $V_i$ is a function $X_1^i$.
We have,
\bas{
\bE[g|X_1^i] = \sum_{x_{i+1}^n\in \cX_{i+1}^n} P(x_{i+1}^n)g(X_1^ix_{i+1}^n),
}
which along with Jensen's inequality gives:

\bas{
e^{\lambda V_i}&=e^{\lambda\sum_{x_i',x_{i+1}^n} P(x_{i+1}^n)P(x'_i)(g(X_1^{i-1},x_i,x_{i}^n)-g(X_1^{i-1},x'_i,x_{i}^n))}\\
&\leq \sum_{x_i',x_{i+1}^n} P(x_{i+1}^n)P(x'_i)e^{\lambda(g(X_1^{i-1},x_i,x_{i}^n)-g(X_1^{i-1},x'_i,x_{i}^n))}\\
\bE\left[e^{\lambda V_i}|X_{1}^{i-1}\right]&\leq \sum_{x_{i+1}^n} P(x_{i+1}^n)\sum_{x_i,x_i'}P(x_i)P(x'_i)e^{\lambda(g(X_1^{i-1},x_i,x_{i}^n)-g(X_1^{i-1},x'_i,x_{i}^n))}\\
}
For fixed $X_1^{i-1}\in \cX_1^{i-1}$ and $x_{i+1}^n\in \cX_{i+1}^n$, define $F_i:\cX_i\to \bR$ by $F_i(y) = g(X_1^{i-1}yx_{i+1}^n)$. Using the definition of $g(X)$ and denoting by $\tilde{\bmu}$ the $\bmu$ that achieves the supremum in $g(X_1^{i-1}yx_{i+1}^n)$, we get: 
\bas{
F_i(y)-F_i(y') =& \sup_{\bmu\in \bA}\left(\frac{1}{n}\sum_{i=1}^n w_1(X_i;\bmu)\ip{X_i-\bmu_1,u}-\bE_X w_1(X;\bmu)\ip{X-\bmu_1,u}\right) \\
&\qquad \qquad - \sup_{\bmu\in \bA} \left(\frac{1}{n}\sum_{i=1}^n w_1(X'_i;\bmu)\ip{X'_i-\bmu_1,u}-\bE_X w_1(X';\bmu)\ip{X'-\bmu_1,u}\right) \\
\le & \left(\frac{1}{n}\sum_{i=1}^n w_1(X_i; \tilde{\bmu})\ip{X_i-\tilde{\bmu}_1,u}-\bE_X w_1(X;\tilde{\bmu})\ip{X- \tilde{\bmu}_1,u}\right) \\
&\qquad \qquad -  \left(\frac{1}{n}\sum_{i=1}^n w_1(X'_i; \tilde{\bmu})\ip{X'_i- \tilde{\bmu}_1,u}-\bE_X w_1(X';\tilde{\bmu})\ip{X'- \tilde{\bmu}_1,u}\right) \\
=& \frac{1}{n} \left( w_1(y; \tilde{\bmu})\ip{y -\tilde{\bmu}_1,u} -  w_1(y'; \tilde{\bmu})\ip{y' - \tilde{\bmu}_1,u}\right)
}
Take $\bar{\bmu}$ as the maximizer for the supremum of $X'$ we get the other side of the inequality. Hence
\bas{
|F_i(y)-F_i(y')| \le& \frac{1}{n} \sup_{\bmu\in \bA} \left( |\ip{y-\bmu,u}| + |\ip{y'-\bmu,u}| \right)\\
\le& \frac{1}{n}\left( |\ip{y-\bmu^*_{Z_y},u}| + |\ip{y'-\bmu^*_{Z_y'},u}| + 4\rmax \right) :=\rho(y,y')
}

Note for all $t$ such that $|t|<s$, we have $e^t+e^{-t}\le e^s+e^{-s}$, we have
\bas{
	e^{\lambda(F(y)-F(y'))}+e^{-\lambda(F(y)-F(y'))} \le e^{\lambda \rho(y,y')}+e^{-\lambda \rho(y,y')}
}

By symmetry,  we have:
\bas{
\sum_{y,y'} P(y)P(y') e^{\lambda(F(y)-F(y'))}
 \le& \frac{1}{2}\left( \bE_{y,y'}e^{\lambda \rho(y,y')}+ \bE_{y,y'}e^{-\lambda \rho(y,y')} \right)\\
 =& \bE_\epsilon \bE_{y,y'}e^{\lambda \epsilon \rho(y,y')}\stackrel{(i)}{\leq} e^{\lambda^2/n^2} E[e^{4\lambda\epsilon\rmax/n}]\\
\stackrel{(ii)}{\leq}  & e^{\lambda^2/n^2+8\lambda^2\rmax^2/n^2}\leq e^{\lambda^2(1+3\rmax)^2/n^2}
}

where $\epsilon$ is a Rademacher random variable independent of $y,y'$. Note that $\epsilon|\ip{y-\bmu^*_{Z_y},u}|$ is identically distributed as a Gaussian random variable with mean zero and variance $1$. Also since by construction $y$ and $y'$ are independent, inequality $(i)$ follows using the moment generating function of a Gaussian. Inequality $(ii)$ follows from Hoeffding's Lemma (Eq (3.16) in \cite{hoeffding1963probability}) since $\epsilon\in[-1,1]$. Therefore,
\ba{
\bE\left[e^{\lambda V_i}|X_{1}^{i-1}\right]&\le e^{\lambda^2/n^2+8\lambda^2\rmax^2/n^2}\leq e^{\lambda^2(1+3\rmax)^2/n^2}
\label{eq:upper_bound_elambdaV}
}
Applying standard Markov inequality, we have
\bas{
P(g(X)-\bE g(X)>t) =& P\left( \sum_{i=1}^n V_i>t \right)
\le  e^{-\lambda t}\bE \left[ \Pi_{i=1}^n e^{\lambda V_i} \right]\\
\stackrel{(i)}{=}& e^{-\lambda t}\bE\left[\bE \left[ \Pi_{i=1}^n e^{\lambda V_i }| X_1^{n-1} \right]\right]\\
=& e^{-\lambda t}\bE\left[ \Pi_{i=1}^{n-1} e^{\lambda V_i }\bE \left[e^{\lambda V_n }| X_1^{n-1} \right]\right]\\
\stackrel{(ii)}{=}&e^{-\lambda t}\bE\left[ \Pi_{i=1}^{n} \bE\left[e^{\lambda V_i }|X_1^{i-1}] \right]\right]\\
\stackrel{(iii)}{\le} & \exp\left( -\lambda t+\frac{\lambda^2(1+3\rmax)^2}{n}\right)
}
where step $(ii)$ follows by applying step $(i)$ repeatedly and step $(iii)$ follows by applying Eq~\eqref{eq:upper_bound_elambdaV}.
Optimizing over $\lambda$ we have 
$P(g(X)-\bE g(X)>t)\le  \exp\left( -\frac{nt^2}{4(1+\rmax)^2} \right)$.
Taking $t = 2(1+3\rmax)\sqrt{\frac{d\log n}{n}}$, we have
\bas{
P \left( g(X)-\bE g(X) > 2(1+3\rmax)\sqrt{\frac{d\log n}{n}} \right) \le & n^{ -d }
}

\end{proof}

With Theorem~\ref{th:extension_aryeh} we are able to improve Theorem~\ref{th:sample_eps} and have the following.

\begin{theorem}[Rate-optimal sample-based EM guarantee]
Denote $\mathbb{A}$ as the contraction region $\Pi_{i=1}^M\bB(\bmu_i^*,a)$.  Under the condition of Theorem \ref{th:population_contraction}, with probability at least $1-\exp\left( -cd \log n\right)$, 
\bas{
\sup_{\bmu \in \mathbb{A}}\norm{G^{(i)}(\bmu)-G_n^{(i)}(\bmu)}<\epsilon^{\text{unif}}(n); \qquad \forall i\in [M]
}
where 
\ba{
\epsilon^{\text{unif}}(n)=cM^{3/2}(1+3\rmax)^3\max\{1,\log(\kappa)\}\sqrt{\frac{d\log n}{n}}.
\label{eq:improved_eps}
}
\label{th:improved_eps}
\end{theorem}

\begin{proof}
Combining Proposition~\ref{prop:rademacher_bound}, Lemma~\ref{lem:symmetrization} and Theorem~\ref{th:extension_aryeh}, we have for any $d$-dimensional unit vector $u$, with probability at least $1-n^{-d}$, 
\bas{
g_i^u(X) \le& |g_i^u(X)-\bE g_i^u(X)|+\bE g_i^u(X)\\
\le & 2(1+3\rmax) \sqrt{\frac{d\log n}{n}} + 2\mathcal{R}_n(\cF_i^u)\\
\le & cM^{3/2}(1+3\rmax)^3\max\{1,\log(\kappa)\}\sqrt{\frac{d\log n}{n}}
}
By the same covering argument as in the proof of Theorem~\ref{th:sample_eps}, we have 
\bas{
\sup_{\bmu \in \mathbb{A}}\norm{G^{(i)}(\bmu)-G_n^{(i)}(\bmu)} \le 2 \max_{j=1,\cdots, K} g_i^{u^{(j)}}(X)
}
Using $K\le e^{2d}$ from Lemma~\ref{lem:covering_number} along with union bound, we have 
\bas{
\sup_{\bmu \in \mathbb{A}}\norm{G^{(i)}(\bmu)-G_n^{(i)}(\bmu)} \le cM^{3/2}(1+3\rmax)^3\max\{1,\log(\kappa)\}\sqrt{\frac{d\log n}{n}}
}
with probability at least $1-(ne^{-2})^{-d}$.

\end{proof}
\section{Initialization}
\label{app:initialization}
This section provides the number of initializations needed for the condition in Theorem~\ref{th:population_contraction}.
\begin{proposition}
Let $\pi_i = \frac{1}{M}, \forall i\in [M]$, $R_{\min}=\Omega(\sqrt{d})$, and let $a$ satisfy the conditions in Theorem~\ref{th:population_contraction}. Then 
with $\frac{\log(1/\delta)}{\sqrt{2\pi M}}\left(\frac{e}{1-e^{-a\sqrt{d}/2}}\right)^M$ initializations, the probability of having at least one good initialization is greater than $1-\delta$.
\label{prop:number_of_iterations}
\end{proposition}
The proof follows directly from some combinatorial arguments and Lemma~\ref{lem:norm_of_gaussian}.

\begin{proof}[Proof of Proposition~\ref{prop:number_of_iterations}]
Define event $\mathcal{E}_{init}(a)=\{ \mu_i^0\in\bB_{\mu_i^*}(a),\  \forall i \in [M] \}$.
By equal weights assumption, the probability of randomly sampled $M$ points having exactly one from each cluster is $\frac{M!}{M^M}$. By Sterling's formula, we have  $M!\geq \sqrt{2\pi M}e^{-M}$.
For each center, by Lemma~\ref{lem:norm_of_gaussian} we have the probability of it lying in $\bB_{\mu_i^*}(a)$ is no less than $1-e^{-a\sqrt{d}/2 }$. Hence 
\bas{
P(\mathcal{E}_{init}(a)) \ge \sqrt{2\pi M}\left(\frac{1-e^{-a\sqrt{d}/2}}{e}\right)^{M}=:p
}
Now assume the number of initializations is $T$, in order to satisfy the required property, we need $(1-P(\mathcal{E}_{init}(a)))^T\le \delta$. A sufficient condition is 
\bas{
T\ge \frac{\log(1/\delta)}{\log\left(1 - p\right)}
}
Note that $\log(1-x)\ge -x, \forall 0\le x\le 0.5$. Since $p<.5$ for $M\geq 2$, we see that as long as 
$T\ge \frac{\log(1/\delta)}{\sqrt{2\pi M}}\left(\frac{e}{1-e^{-a\sqrt{d}/2}}\right)^M$, with probability $1-\delta$ we will have a good initialization.
\end{proof}
\begin{remark}
	Perhaps not so surprisingly, the above theorem requires a stronger separation condition, i.e. $\rmin=\Omega(\sqrt{d})$, whereas all our analysis requires $\rmin=\Omega(\sqrt{d_0})$ where $d_0:=\min(d,M)$ can be thought of as effective dimension. This difficulty can be alleviated by using projections schemes similar to those in~\cite{awasthi2012improved,kumar2010clustering}. We leave this for future work.
	\end{remark}
\bk

\bibliographystyle{plain}
\bibliography{reference}

\end{document}